\documentclass[11pt,a4paper,reqno,draft]{amsart}

\usepackage{amssymb,amsmath}
\usepackage{amsthm}
\usepackage{latexsym}
\usepackage[colorlinks=true,linktocpage=true,pagebackref=false, citecolor=black,linkcolor=black]{hyperref}
\usepackage{graphics}
\usepackage{euscript}
\usepackage{mathrsfs}
\usepackage[dvips]{curves}
\usepackage{tikz}
\usetikzlibrary{arrows,decorations.pathmorphing,backgrounds,positioning,fit,matrix}
\usepackage{comment}
\usepackage{multirow}
\usepackage{xcolor}
\usepackage[all]{xy}
\usepackage{caption,subcaption}
\usepackage{pgfplots}
\usepackage{mathabx}

\pgfplotsset{soldot/.style={color=black,only marks,mark=*}} \pgfplotsset{holdot/.style={color=black,fill=white,only marks,mark=*}}

\newtheorem{thm}{Theorem}

\newtheorem{prop}[thm]{Proposition}
\newtheorem{cor}[thm]{Corollary}
\newtheorem{lem}[thm]{Lemma}

\newtheorem{prob}[thm]{Problem}

\theoremstyle{definition}

\theoremstyle{remark}
\newtheorem{rem}[thm]{Remark}

\newtheorem{exa}[thm]{Example}

 \newcommand{\N}{{\mathbb N}}
\newcommand{\Z}{{\mathbb Z}} \newcommand{\R}{{\mathbb R}}
 \newcommand{\C}{{\mathbb C}}

\newcommand{\sph}{{\mathbb S}}

\newcommand{\pol}{{\EuScript K}}

\newcommand{\tildebaja}{{\raise.17ex\hbox{$\scriptstyle\sim$}}}

\newcommand{\Int}{\operatorname{Int}}
\newcommand{\im}{\operatorname{Im}}
\newcommand{\cl}{\operatorname{Cl}}

\newcommand{\dist}{\operatorname{dist}}
\newcommand{\id}{\operatorname{id}}

\newcommand{\zar}{\operatorname{zar}}

\newcommand{\Reg}{\operatorname{Reg}}
\newcommand{\Sing}{\operatorname{Sing}}

\newcommand{\x}{{\tt x}}

\newcommand{\gtm}{{\mathfrak m}} \newcommand{\gtn}{{\mathfrak n}}

 \newcommand{\gtM}{{\mathfrak M}}

\newcommand{\ol}{\overline}

\newcommand{\veps}{\varepsilon}

\newcommand{\sfs}{\mathsf{s}}
\newcommand{\sft}{\mathsf{t}}

\begin{document}

\title[Differentiable approximation of continuous semialgebraic maps]{Differentiable approximation of continuous semialgebraic maps}

\author{Jos\'e F. Fernando}
\address{Departamento de \'Algebra, Geomatr\'\i a y Topolog\'\i a, Facultad de Ciencias Matem\'aticas, Universidad Complutense de Madrid, 28040 MADRID (SPAIN)}
\email{josefer@mat.ucm.es}
\thanks{The first author is supported by Spanish GRAYAS MTM2014-55565-P, Spanish STRANO MTM2017-82105-P and Grupos UCM 910444. This article has been mainly written during a couple of one month research stays of the first author in the Dipartimento di Matematica of the Universit\`a di Trento. The first author would like to thank this department for the invitation and the very pleasant working conditions.}

\author{Riccardo Ghiloni}
\address{Dipartimento di Matematica, Via Sommarive, 14, Universit\`a di Trento, 38123 Povo (ITALY)}
\email{riccardo.ghiloni@unitn.it}
\thanks{The second author is supported by GNSAGA of INDAM} 

\date{}



\maketitle

\begin{abstract}
In this work we approach the problem of approximating uniformly continuous semialgebraic maps $f:S\to T$ from a compact semialgebraic set $S$ to an arbitrary semialgebraic set $T$ by semialgebraic maps $g:S\to T$ that are differentiable of class~${\mathcal C}^\nu$ for a fixed integer $\nu\geq1$. As the reader can expect, the difficulty arises mainly when one tries to keep the same target space after approximation. For $\nu=1$ we give a complete affirmative solution to the problem: such a uniform approximation is always possible. For $\nu \geq 2$ we obtain density results in the two following relevant situations: either $T$ is compact and locally ${\mathcal C}^\nu$ semialgebraically equivalent to a polyhedron, for instance when $T$ is a compact polyhedron; or $T$ is an open semialgebraic subset of a Nash set, for instance when $T$ is a Nash set. Our density results are based on a recent ${\mathcal C}^1$-triangulation theorem for semialgebraic sets due to Ohmoto and Shiota, and on new approximation techniques we develop in the present paper. Our results are sharp in a sense we specify by explicit examples.


\end{abstract}


\section{Introduction and main results}\label{s1}

The importance of approximation of continuous functions is a natural question that arises from Stone-Weierstrass result of uniform approximation of continuous functions over compact sets by polynomial functions. This result provides naturally a polynomial approximation result for $\R^n$-valued continuous maps $f:K\to\R^n$ defined on a compact set~$K$. Difficulties arise when trying to restrict the image of the approximating map. One way to proceed is to be more flexible with the type of approximating maps but also by considering tame sets as domain of definition and target space. When considering suitable classes of approximating maps with some kind of regularity (and not only polynomials), we can also relax the compactness assumption concerning the domain of definition. A crucial fact is the topology we provide the space of maps we are working with because it determines the kind of approximation result we can achieve. Namely, given a space of maps ${\mathcal F}(X,Y)$ between two sets $X$ and $Y$ endowed with a certain topology and a subset ${\mathcal G}(X,Y)$, each map in ${\mathcal F}(X,Y)$ can be approximated by maps in ${\mathcal G}(X,Y)$ if and only if ${\mathcal G}(X,Y)$ is dense in ${\mathcal F}(X,Y)$.

In order to state properly the main results collected in this section, we need some preliminary definitions. Recall that a set $S\subset\R^m$ is \em semialgebraic \em if it is a Boolean combi\-nation of sets defined by polynomial equalities and inequalities. A map $f:S\to T$ between semialgebraic sets $S\subset\R^m$ and $T\subset\R^n$ is \em semialgebraic \em if its graph is a semialgebraic subset of $\R^{m+n}$. Let $\Omega$ be a (non-empty) open semialgebraic subset of $\R^m$, let $\nu\in\N\cup\{\infty\}$ and let ${\mathcal S}^\nu(\Omega):={\mathcal S}^\nu(\Omega,\R)$ be the set of all (real-valued) semialgebraic functions defined on $\Omega$ that are differentiable of class~${\mathcal C}^\nu$. If $\nu=\infty$, the functions in ${\mathcal S}^\infty(\Omega)$ are called \em Nash functions \em on $\Omega$, and the set ${\mathcal S}^\infty(\Omega)$ is usually denoted by ${\mathcal N}(\Omega)$. We equip ${\mathcal S}^\nu(\Omega)$ with the following \em ${\mathcal S}^\nu$ $($Whitney$\,)$ topology\em, which makes it a Hausdorff topological ring \cite[II.1]{sh}. For each positive continuous semialgebraic function $\veps:\Omega\to\R$ and for each multi-index $\alpha:=(\alpha_1,\ldots,\alpha_m)\in\N^m$ with $|\alpha|:=\sum_{i=1}^m\alpha_i\leq\nu$, we define the set
\[
{\mathcal U}_{\veps,\alpha}:=\big\{g\in {\mathcal S}^\nu(\Omega):\, |D_\alpha\,g|<\veps\big\},
\]
where $D_\alpha\,g$ denotes the partial derivative $\partial^{|\alpha|}g/\partial\x_1^{\alpha_1}\cdots\partial\x_m^{\alpha_m}$. The sets ${\mathcal U}_{\veps,\alpha}$ form a subbase of neighborhoods of the zero function for the above-mentioned topology. A semialgebraic function $f:S\to\R$ defined on the semialgebraic set $S\subset\R^m$ is ${\mathcal S}^\nu$ $($or \em Nash \em if $\nu=\infty)$ if there exist an open semialgebraic neighborhood $U$ of $S$ in $\R^m$ and a function $F\in{\mathcal S}^\nu(U)$ extending $f$ to $U$. We denote the set of these functions by ${\mathcal S}^\nu(S)$ (or by ${\mathcal N}(S)$ if $\nu=\infty$).

Suppose now that $\nu\in\N$ and $S\subset\R^m$ is \em locally compact\em. The latter condition on $S$ is equivalent to the existence of an open semialgebraic neighborhood $\Omega$ of $S$ in $\R^m$ such that $S$ is \em closed \em in $\Omega$. In the ${\mathcal S}^\nu$ category there are bump functions as well as finite partitions of unity \cite[II.2]{sh}. Using this it follows, due to the closedness of $S$ in $\Omega$, that the restriction map $\rho_{\scriptscriptstyle\Omega}\!:{\mathcal S}^\nu(\Omega)\to{\mathcal S}^\nu(S)$ is surjective. We can equip the set ${\mathcal S}^\nu(S)$ with the quotient topology $\tau_\Omega^\nu$ induced by $\rho_{\scriptscriptstyle\Omega}$ and consequently $\rho_{\scriptscriptstyle\Omega}$ is an open (continuous) map and ${\mathcal S}^\nu(S)$ is a Haudorff topological ring. Such a quotient topology on ${\mathcal S}^\nu(S)$ is called \em ${\mathcal S}^\nu$~topology \em (induced by $\rho_{\scriptscriptstyle\Omega}$). This topology does not depend on the chosen open semialgebraic set $\Omega$ that contains $S$.
\begin{proof}
Let $\Omega'\subset\Omega$ be an open semialgebraic neighborhood of $S$ in $\Omega$. We claim: $\tau_\Omega^\nu=\tau_{\Omega'}^\nu$.

Consider an ${\mathcal S}^\nu$ partition of unity $\theta_1,\theta_2:\Omega\to\R$ subordinated to the open semialgebraic covering $\{\Omega\setminus S,\Omega'\}$. The homomorphisms $\varPsi_{\Omega\Omega'}:{\mathcal S}^\nu(\Omega)\to{\mathcal S}^\nu(\Omega'): f\mapsto(\theta_2 f)|_{\Omega'}$ and $\varPsi_{\Omega'\Omega}:{\mathcal S}^\nu(\Omega')\to {\mathcal S}^\nu(\Omega):g\mapsto\theta_2 g$ are continuous. The ${\mathcal S}^\nu$ function $\theta_2g$ values $\theta_2(x)g(x)$ if $x\in\Omega'$ and $0$ if $x\in\Omega\setminus\Omega'$. The restrictions of $(\theta_2f)|_{\Omega'}$ and $\theta_2g$ to $S$ value respectively $f|_S$ and $g|_S$. Thus, we have the commutative diagram 
\[
\xymatrix{{\mathcal S}^\nu(\Omega)\ar[d]_{\rho_\Omega}\ar@<0.5ex>[r]^{\varPsi_{\Omega\Omega'}}&{\mathcal S}^\nu(\Omega')\ar@<0.5ex>[l]^{\varPsi_{\Omega'\Omega}}\ar[d]^{\rho_{\Omega'}}\\ 
{\mathcal S}^\nu(S)\ar@<0.5ex>[r]^{\text{Id}}&{\mathcal S}^\nu(S)\ar@<0.5ex>[l]^{\text{Id}} } 
\]
where $\rho_\Omega:{\mathcal S}^\nu(\Omega)\to {\mathcal S}^\nu(S) $ and $\rho_{\Omega'}:{\mathcal S}^\nu(\Omega')\to{\mathcal S}^\nu(S)$ are the open quotient homomorphisms. The identity map $\text{Id}$ from left to right is continuous if and only if $\text{Id}\circ\rho_\Omega$ is continuous if and only if $\rho_{\Omega'}\circ \varPsi_{\Omega\Omega'}$ is continuous, which is true. The continuity of the identity map from right to left is similar. Consequently, the topology does not depend on the neighborhood. The previous argument, which works for the general locally compact case, appears in \cite[p.\hspace{2pt}75]{bfr} implemented only for the particular case of Nash manifolds with corners.
\end{proof}

In fact, one can show using \cite[Prop.2.C.3]{bfr} that the ${\mathcal S}^\nu$ topology of ${\mathcal S}^\nu(S)$ does not depend on the ${\mathcal S}^\nu$ immersion of $S$ as a closed subset of an open semialgebraic subset of an arbitrary Euclidean space.

By \cite{dk} the space ${\mathcal S}^0(S)$ (for $\nu=0$) coincides with the topological ring of all continuous semialgebraic functions on $S$ endowed with the topology that has as a basis of open neighborhoods of the zero function the family constituted by the sets $\{g\in {\mathcal S}^0(S):\, |g|<\veps\}$, where $\veps:S\to\R$ is any positive continuous semialgebraic function.

A semialgebraic map $f:=(f_1,\ldots,f_n):S\to T\subset\R^n$ between semialgebraic sets is said to be ${\mathcal S}^\nu$ (or \em Nash \em if $\nu=\infty$) if its components $f_i:S\to\R$ are ${\mathcal S}^\nu$ functions. We denote ${\mathcal S}^\nu(S,T)$ the set of these maps (or ${\mathcal N}(S,T)$ if $\nu=\infty$). If $\nu\in\N$, we endow ${\mathcal S}^\nu(S,T)$ with the subspace topology induced by the product topology of ${\mathcal S}^\nu(S,\R^n)=({\mathcal S}^\nu(S,\R))^n$. We call such topology the \em ${\mathcal S}^\nu$ topology \em of ${\mathcal S}^\nu(S,T)$. 

Let $\nu\geq1$. A semialgebraic set $M\subset\R^m$ is called an (\emph{affine}) \em ${\mathcal S}^\nu$~manifold \em (or a \em Nash manifold \em if $\nu=\infty$) if it is in addition a ${\mathcal C}^\nu$ submanifold of (an open subset of) $\R^m$. As in this paper all manifolds are affine we often drop the adjective affine. A~map $f:M\to N$ between ${\mathcal S}^\nu$ manifolds is ${\mathcal S}^\nu$ in the sense of the above paragraph if and only if it is semialgebraic and differentiable of class ${\mathcal C}^\nu$ in the usual sense for ${\mathcal C}^\nu$ manifolds involving charts. For more details concerning the spaces of ${\mathcal S}^\nu$ maps, we refer the reader to \cite[2.C~\&~2.D]{bfr}.

We recall also that if $U$ is an open semialgebraic subset of $\R^n$, a set $Y\subset U$ is called \em Nash subset of $U$ \em if there exists a Nash function $g\in {\mathcal N}(U)$ such that $Y$ is the zero locus of $g$. A \em Nash set \em is a Nash subset of an open semialgebraic subset of some $\R^n$. Naturally, a \em Nash subset \em $Z$ of a Nash set $Y$ is a Nash set $Z$ that is closed in $Y$. Observe that a Nash set is semialgebraic and a Nash manifold is a (non-singular) Nash set.

\subsection{Some relevant known approximation results}
There are many approximation results in the literature and we recall here some of them.

\subsubsection*{$\mathcal{S}^\nu$ maps by Nash maps} 
Efroymson's approximation theorem \cite[\S1]{ef} ensures that conti\-nuous semialgebraic functions can be approximated by Nash functions on a Nash manifold. This statement was improved by Shiota in many directions \cite{sh}, for instance, providing a similar approximation result for ${\mathcal S}^\nu$ functions using the ${\mathcal S}^\nu$ topology and proving relative versions of such a ${\mathcal S}^\nu$ approximation result. The previous results can be extended to approximate ${\mathcal S}^\nu$ maps $f:S\to N$ from a locally compact semialgebraic set $S\subset\R^m$ to a Nash manifold $N\subset\R^n$ by Nash maps $g:S\to N$ in the ${\mathcal S}^\nu$ topology, making use of a suitable Nash tubular neighborhood of $N$ in $\R^n$ (see \cite[Lem.I.3.2]{sh}). We can even go further and obtain approximation results for ${\mathcal S}^\nu$ maps between Nash sets with monomial singularities.

\begin{thm}[{\cite[Thm.1.7]{bfr}}]\label{thm17}
Let $M$ and $N$ be Nash manifolds, let $X\subset M$ be a Nash set with monomial singularities and let $Y\subset N$ be a Nash set with monomial singularities whose irreducible components are non-singular. Define $m:=\dim(M)$, $n:=\dim(N)$ and $q:=m\big(\binom{n}{[n/2]}-1\big)$ where $[n/2]$ denotes the integer part of $n/2$. If $\nu\ge\! q$ then every ${\mathcal S}^\nu$ map $f:X\to Y$ that preserves irreducible components can be ${\mathcal S}^{\nu-q}$ approximated by Nash maps $g:X\to Y$.
\end{thm}

As a by-product, one has the following classification result for Nash manifolds with corners. Recall that a Nash manifold with corners $Q$ has \em divisorial corners \em if it is contained in a Nash manifold $M$ of its same dimension and the Nash closure of $\partial Q$ in $M$ is a normal-crossings divisor (as one can expect this is not always the case and a careful study can be found in \cite[1.12]{fgr}).

\begin{thm}[{\cite[Thm.1.8]{bfr}}]\label{difcor}
Let $Q_1$ and $Q_2$ be two $m$-dimensional Nash manifolds with divisorial corners. If $Q_1$ and $Q_2$ are ${\mathcal S}^\nu$ diffeomorphic for some $\nu>m^2$ then they are Nash diffeomorphic. 
\end{thm}

\subsubsection*{Nash maps by regular maps}

The problem of approximating (smooth or) Nash maps between real algebraic manifolds by regular maps is an old and deep question in real algebraic geometry. Let $X$ and $Y$ be real algebraic manifolds of positive dimension such that $X$ is compact. The set ${\mathcal R}(X,Y)$ of regular maps from $X$ to $Y$ turns out to be dense in the corresponding space ${\mathcal N}(X,Y)$ of Nash maps endowed with the ${\mathcal C}^\infty$ compact-open topolo\-gy only in exceptional cases. Besides the Stone-Weierstrass theorem (quoted above) for which $Y$ is an Euclidean space, the density of ${\mathcal R}(X,Y)$ in ${\mathcal N}(X,Y)$ is known only when the target space $Y$ is one of the spheres $\sph^1$, $\sph^2$, $\sph^4$ or a grassmannian. For a general rational real algebraic manifold $Y$, we must restrict hardly the possible domains of defi\-nition $X$ to some special types. A remarkable example is the density of ${\mathcal R}(\sph^m,\sph^n)$ in ${\mathcal N}(\sph^m,\sph^n)$ for each pair $(m,n)$ when $n=1,2,4$. If $n\neq1,2,4$, the density of ${\mathcal R}(\sph^m,\sph^n)$ in ${\mathcal N}(\sph^m,\sph^n)$ remains a fascinating mystery. For further details, we refer the reader to \cite[Ch.12 \& \S13.3]{bcr} and the quoted references there, to the survey \cite{bk10} and to the articles \cite{bk87,bk88,bk89,bk93,bk07,bks97,gh07,kz10}. In \cite{gh06-1,gh06-2} the second author showed that if $Y$ is `generic' in a suitable way, ${\mathcal R}(X,Y)$ is an `extremely small' closed subset of ${\mathcal N}(X,Y)$. This lack of regular maps between real algebraic manifolds seems to be the main obstru\-ction for an extension of the Nash-Tognoli algebraization techniques from smooth manifolds to singular polyhedral spaces and, in particular, to compact Nash sets \cite{ak,gt}.

\subsubsection*{Continuous maps by continuous rational maps.} 
Kucharz has studied deeply approximation results of continuous maps between a compact algebraic manifold $X$ and a sphere $\sph^n$ by continuous rational maps. As $X$ is compact the author considers on the space ${\mathcal C}(X,\sph^n)$ of continuous maps from $X$ to $\sph^n$ the compact-open topology. In case $X$ has dimension $n$, the space ${\mathcal R}_0(X,\sph^n)$ of (nice) continuous rational maps from $X$ to $\sph^n$ is dense in ${\mathcal C}(X,\sph^n)$ (see \cite[Thm.1.2, Cor.1.3]{kz} for further details). In addition, for any pair $(m,n)$ of non-negative integers, the set ${\mathcal R}_0(\sph^m,\sph^n)$ is dense in ${\mathcal C}(\sph^m,\sph^n)$ (see \cite[Thm.1.5]{kz}). In general ${\mathcal R}_0(X,\sph^n)$ needs not to be dense in ${\mathcal C}(X,\sph^n)$. Simple obstructions can be expressed in terms of homology or cohomology classes representable by algebraic subsets. We refer the reader to \cite{kz2,kz} for further details.

\subsubsection*{Homeomorphisms between smooth manifolds by diffeomorphisms.}
There are many and celebrated approximation results concerning homeomorphisms between smooth manifolds by diffeomorphisms in the literature. The obstruction theory originated from the problem of smoothing a continuous map with good properties or smoothing a combinatorial manifold $M$ deserves special attention because it is in the core of differential topology. Such a theory was mainly developed by Milnor, Thom, Munkres and Hirsch. They found that the obstructions concentrate in certain homology classes belonging to the homology groups of the combinatorial manifold $M$ relative to its boundary $\partial M$ with coefficients in the quotient groups of the orientation-preserving diffeomorphisms of the unit spheres $\sph^p$ modulo the orientation-preserving diffeomorphisms of the unit balls ${\mathbb B}^{p+1}$. We refer the reader to \cite{hi1,hi2,mi1,mu1,mu2,mu3,th} for further details concerning the mentioned foundational obstruction results and to \cite{DP,HM,IKO,MP,Mu} for some recent developments. 

In addition, Milnor \cite{mi2} discovered two compact polyhedra which are homeomorphic but not piecewise linearly homeomorphic (a counterexample to the Hauptvermutung). In the case in which the homeomorphism is semialgebraic the situation is completely different. In fact, Shiota and Yokoi \cite{shy} proved, using approximation techniques, that two semialgebraically homeomorphic compact polyhedra in $\R^n$ are also piecewise linearly homeomorphic. In \cite{sh3} Shiota improved the previous result and he obtained a PL homeomorphism by a constructive procedure, involving more sophisticated approximation techniques, that starts from the original homeomorphism. He proved that, for any ordered field $R$ equipped with any o-minimal structure, two definably homeomorphic compact polyhedra in $R^n$ are PL homeomorphic (the o-minimal Hauptvermutung). Together with the fact that any compact definable set is definably homeomorphic to a compact polyhedron, he concludes that o-minimal topology is `tame'.

\subsection{Our approximation results in the semialgebraic setting.}
The general problem regarding approximation of maps in the semialgebraic setting can be stated as follows:

\begin{prob}\label{prob1}
Let $S\subset\R^m$ and $T\subset\R^n$ be semialgebraic sets and let $\ell,\nu\in\N\cup\{\infty\}$ be such that $S$ is locally compact and $\ell<\nu$. Is ${\mathcal S}^\nu(S,T)$ dense in ${\mathcal S}^\ell(S,T)$ endowed with the ${\mathcal S}^\ell$~topology?
\end{prob}

It is well-known that the answer is affirmative if the target space is a ${\mathcal S}^\nu$ manifold.

\begin{thm}[{\cite[I \& II]{sh}}] \label{kr1}
For each $\ell,\nu\in\N\cup\{\infty\}$ with $\ell<\nu$, if $T$ is a ${\mathcal S}^\nu$ manifold, then ${\mathcal S}^\nu(S,T)$ is dense in ${\mathcal S}^\ell(S,T)$.
\end{thm}

The reason why the latter result works is that each ${\mathcal S}^\ell$ map $f:S\to T\subset\R^n$ can be ${\mathcal S}^\ell$ approximated by a ${\mathcal S}^\nu$ map $g:S\to\R^n$ (see \cite[II.4.1]{sh}) and then one can use a ${\mathcal S}^\nu$~(bent) tubular neighborhood $\rho:U\to T$ of $T$ in $\R^n$ (see \cite[I.3.5 \& II.6.1]{sh}) to obtain the desired approximating map $\rho\circ g$. However, an arbitrary semialgebraic set $T\subset\R^n$ does not have ${\mathcal S}^\nu$ tubular neighborhoods in $\R^n$ if $\nu\geq1$ (see Example \ref{exa:no-retraction} below). It has only ${\mathcal S}^0$~tubular neighborhoods in $\R^n$, provided it is locally compact \cite{dk}.

The following example shows that Problem \ref{prob1} does not have an affirmative solution in general if $\ell\geq1$ and $T$ is not a ${\mathcal S}^\nu$ manifold, even if $S$ is a compact real algebraic manifold and $T$ is a compact real algebraic set.

\begin{exa}
Let $S\subset\R^m$ be a Nash manifold (for instance a compact real algebraic manifold). Assume $S$ meets transversally the hyperplane $\{x_1=0\}$ of $\R^m$ at a point $y_0\in S$. Let $T$ be the Nash set
\[
\{(x_1,\ldots,x_m,x_{m+1})\in S\times\R:\, x_1^{3\ell+1}-x_{m+1}^3=0\},
\]
where $\ell\geq1$ is a fixed positive integer. Observe that $T$ is a compact real algebraic set if $S$ is a compact real algebraic manifold. Choose the integer $\nu:=\ell+1$ and denote $x:=(x_1,\ldots,x_m)$. Observe that $T$ is a ${\mathcal S}^\ell$ manifold, but not a ${\mathcal S}^\nu$ manifold. Indeed, $S$ is not ${\mathcal S}^\nu$ smooth locally at $y_0$. Consider the ${\mathcal S}^\ell$ map $f:S\to T$, $x\mapsto(x,x_1^{\ell+1/3})$. Such a map cannot be ${\mathcal S}^\ell$ approximated by maps in ${\mathcal S}^\nu(S,T)$.

Suppose on the contrary that there exists a ${\mathcal S}^\nu$ map $g:=(g_1,\ldots,g_m,g_{m+1}):S\to T$ arbitrarily close to $f$ in the ${\mathcal S}^\ell$ topology. The projection $\pi:T\to S,\ (x,x_{m+1})\mapsto x$ is a Nash map. By \cite[Prop.2.D.1]{bfr}, the homomorphism $\pi_*:{\mathcal S}^\ell(S,T)\to{\mathcal S}^\ell(S,S),\ h\mapsto\pi\circ h$ is continuous for the ${\mathcal S}^\ell$ topologies. Thus, if $g$ is ${\mathcal S}^\ell$ close enough to $f$ (and hence ${\mathcal S}^1$ close to $f$), the map $g_*:=(g_1,\ldots,g_m):S\to S$ is a ${\mathcal S}^\nu$ diffeomorphism (using semialgebraic inverse function theorem \cite[Prop.2.9.7]{bcr}). As $g_{m+1}=g_1^{\ell+1/3}$, it follows that $(g_{m+1}\circ g_*^{-1})(x)=x_1^{\ell+1/3}$ is a ${\mathcal S}^\nu$ function on $S$, which is a contradiction. Indeed, the function $S \to \R$, $x \mapsto x_1^{\ell+1/3}$ is not ${\mathcal S}^\nu$ locally at the point $y_0\in S$. This proves that ${\mathcal S}^\nu(S,T)$ is not dense in ${\mathcal S}^\ell(S,T)$ endowed with the ${\mathcal S}^\ell$ topology. $\sqbullet$
\end{exa}

In this paper we deal with Problem \ref{prob1} in the case $\ell=0$ and $S$ is compact. 

\subsubsection*{First main result} Our first result gives a complete affirmative solution to the mentioned version of Problem \ref{prob1} for $\nu=1$. In what follows we indicate $\|x\|_n$ the Euclidean norm of the vector $x$ of $\R^n$.

\begin{thm}\label{thm:C1}
Let $S\subset\R^m$ be a compact semialgebraic set and let $T\subset\R^n$ be a semialgebraic set. Then ${\mathcal S}^1(S,T)$ is dense in ${\mathcal S}^0(S,T)$. More precisely, given any $n\in\N$, the following assertion holds: for each $f\in{\mathcal S}^0(S,\R^n)$ and each $\veps>0$, there exists $g\in{\mathcal S}^1(S,\R^n)$ such that $g(S)\subset f(S)$ and $\|g(x)-f(x)\|_n<\veps$ for every $x\in S$.
\end{thm}

For an arbitrary positive integer $\nu\geq2$ we obtain density results in two very significant situations we are going to present.

\subsubsection*{Second main result} 
Let us recall the definition of locally ${\mathcal S}^\nu$ polyhedral semialgebraic set, which represents a polyhedral counterpart of the concept of ${\mathcal S}^\nu$ manifold. Let $\nu\geq1$ be an integer. A~semialgebraic set $T\subset\R^n$ is called \em locally ${\mathcal S}^\nu$ equivalent to a polyhedron\em, or {\em locally ${\mathcal S}^\nu$~polyhedral} for short, if for each point $x\in T$ there exist two open semialgebraic neighborhoods $U_x$ and $V_x$ of $x$ in $\R^n$, a ${\mathcal S}^\nu$ diffeomorphism $\phi_x:U_x\to V_x$ and a compact polyhedron $Q$ of $\R^n$ such that $\phi_x(U_x\cap T)=V_x\cap Q$. The term \textit{compact polyhedron of $\R^n$} means the realization of a finite simplicial complex of $\R^n$ (see \cite[\S 2]{mu4}). The importance of locally ${\mathcal S}^\nu$~polyhedral semialgebraic sets is that if they are in addition compact, they admit ${\mathcal S}^\nu$ triangulations (see \S \ref{subsec:triangulation} below). 

Our second main result asserts that, if $\ell=0$, Theorem \ref{kr1} extends to the case in which the domain of definition $S$ is compact and the target space $T$ is an arbitrary locally ${\mathcal S}^\nu$ polyhedral compact semialgebraic set.

\begin{thm}\label{main2}
Let $S\subset\R^m$ and $T\subset\R^n$ be compact semialgebraic sets. If $T$ is locally ${\mathcal S}^\nu$~polyhedral for some integer $\nu\geq1$, then ${\mathcal S}^\nu(S,T)$ is dense in ${\mathcal S}^0(S,T)$. 
\end{thm}

As an immediate consequence we obtain:

\begin{cor}\label{cor:PQ}
Let $S\subset\R^m$ be a compact semialgebraic set and let $T\subset\R^n$ be a compact polyhedron. Then ${\mathcal S}^\nu(S,T)$ is dense in ${\mathcal S}^0(S,T)$ for each integer $\nu\geq1$.
\end{cor}

Theorem \ref{main2} and Corollary \ref{cor:PQ} are sharp in the sense that they are false if the approxi\-mating maps are Nash, that is, if $\nu=\infty$.

\begin{exa}\label{sameexample}
Let $S:=[-1,1]$ and $T:=\{(x,y)\in[-2,2]\times\R:\,x^2-y^2=0\}$. Observe that $T$ is a compact polyhedron and $T':=T\cap ((-2,2)\times\R)$ is a Nash set whose irreducible components are $T'_\pm:=\{(x,y)\in(-2,2)\times\R:\,x\pm y=0\}$. Consider the ${\mathcal S}^0$ map 
\[
f:S\to T'\subset T,\ t\mapsto(t,|t|)
\] 
and suppose there exists a Nash map $g:S\to T$ close to $f$ in the ${\mathcal S}^0$ topology of ${\mathcal S}^0(S,T)$. Since $g$ is close to $f$ and $f(S)\subset T'$, we may assume $g(S)\subset T'$. As $S$ is an irreducible semialgebraic set in the sense of \cite[\S3]{fg}, the image $g(S)\subset T'$ must be an irreducible semialgebraic set \cite[3.1(iv)]{fg}, so $g(S)$ must be contained either in $T_+'$ or in $T_-'$, which is impossible because $g$ is close to $f$ and $\dim(f(S)\cap T_\pm')=1$. This proves that ${\mathcal N}(S,T)$ is not dense in ${\mathcal S}^0(S,T)$. $\sqbullet$
\end{exa}

A by-product of the argument we will use to prove Theorem \ref{main2} is the following. 

\begin{cor}\label{cor:KP}
Let $K$ be a finite simplicial complex of $\R^p$ and let $P\subset\R^p$ be the corresponding compact polyhedron $|K|$. Then, for each integer $\nu\geq1$, there exists a sequence $\{\iota^\nu_n\}_{n\in\N}$ in ${\mathcal S}^\nu(P,P)$ with the following universal property: if $f$ is a real-valued function in ${\mathcal S}^0(P)$ such that $f|_{\sigma}\in {\mathcal S}^\nu(\sigma)$ for each $\sigma\in K$, then the sequence $\{f\circ\iota^\nu_n\}_{n\in\N}$ is contained in ${\mathcal S}^\nu(P)$ and converges to $f$ in ${\mathcal S}^0(P)$. In particular, the sequence $\{\iota^\nu_n\}_{n\in\N}$ converges to the identity map in ${\mathcal S}^0(P,P)$.
\end{cor}

\subsubsection*{Third main result} Consider the compact algebraic curve $T:=\{y^2-x^3(1-x)=0\}\subset\R^2$. It has a cusp at the origin, so it is not locally ${\mathcal S}^1$ polyhedral. Thus, Theorem~\ref{main2} does not apply if $T$ is the target space. However, Problem \ref{prob1} continues to have an affirmative solution if $\ell=0$ and $S$ is compact because $T$ is a Nash set. More precisely, we are able to prove the following result.

\begin{thm}\label{main1}
Let $S\subset\R^m$ be a compact semialgebraic set and let $T$ be an open semialgebraic subset of a Nash set $Y\subset\R^n$. Then ${\mathcal S}^\nu(S,T)$ is dense in ${\mathcal S}^0(S,T)$ for each integer $\nu\geq1$.
\end{thm}

Theorem \ref{main1} is again sharp in the sense that it is false if the approximating maps are Nash, that is, if $\nu=\infty$: consider the ${\mathcal S}^0$ map $S\to T'$, $t\mapsto (t,|t|)$ of Example \ref{sameexample} and observe that it cannot be approximated by Nash maps between $S$ and $T'$.

\begin{rem}
In the statements of Theorems \ref{thm:C1}, \ref{main2} and \ref{main1} and of Corollary \ref{cor:PQ}, we may assume that the ${\mathcal S}^\nu$ approximating maps are ${\mathcal S}^0$ homotopic to the original ${\mathcal S}^0$ maps. This holds because in Theorem \ref{thm:C1} $f(S)$ is compact and in the remaining results $T$ is locally compact. Hence $f(S)$ and $T$ admit by \cite{dk} ${\mathcal S}^0$ tubular neighborhoods in $\R^n$.~$\sqbullet$
\end{rem}

\subsubsection*{Involved tools} 
The proofs of Theorems \ref{thm:C1} and \ref{main2} are based on the use of ${\mathcal S}^1$ triangulations of arbitrary compact semialgebraic sets \cite[Thm.1.1]{os} and ${\mathcal S}^\nu$ triangulations of locally ${\mathcal S}^\nu$ polyhedral compact semialgebraic sets \cite[Prop.I.3.13 \& Rmk.I.3.22]{sh2} combined with simplicial approximation of continuous maps between compact polyhedra \cite[Ch.2]{mu4} and with a `shrink-widen' approximation technique introduced in Section \ref{s2'}. Corollary \ref{cor:KP} is a consequence of such technique.

The proof of Theorem \ref{main1} involves the use of \em ${\mathcal S}^\nu$ weak retractions \em that are developed in Section \ref{s3}. Let $M\subset\R^m$ be a Nash manifold and let $X\subset M$ be a Nash normal-crossings divisor. Let $W\subset M$ be an open semialgebraic neighborhood of~$X$. ${\mathcal S}^\nu$ weak retractions are ${\mathcal S}^\nu$ maps $\rho:W\to X$ that are ${\mathcal S}^0$ close to the identity on~$X$. In Proposition \ref{wr} we prove the existence of ${\mathcal S}^\nu$ weak retractions. We combine this tool with a strategy employed in \cite[Lem.2.2]{br} that involves resolution of singularities. The use of ${\mathcal S}^\nu$ weak retractions instead of usual retractions is justified by the following example.

\begin{exa}\label{exa:no-retraction}
There exists no ${\mathcal S}^1$ retraction from a semialgebraic neighborhood $U$ of $T:=\{xy=0\}\subset\R^2$ onto $T$. Suppose that $\rho:U\to T$ is such a ${\mathcal S}^1$ retraction. As $\rho|_T$ is equal to the identity map $\id_T$ on $T$, we deduce that $d_0\rho=\id_{\R^2}$. Consequently, there exist open semialgebraic neighborhoods $U_1$ and $U_2$ of the origin in $\R^2$ such that the restriction $\rho|_{U_1}:U_1\to U_2 \subset T$ is a ${\mathcal S}^1$ diffeomorphism, which is a contradiction. $\sqbullet$
\end{exa}


\subsection*{Structure of the article}
All basic notions and preliminary results used in this paper are presented in Section \ref{s2}. The reader can proceed directly to Section \ref{s2'} and refer to Section~\ref{s2} when needed. In Section \ref{s2'} we describe our `shrink-widen' approximation technique and we prove Theorems \ref{thm:C1} and \ref{main2}, and Corollary \ref{cor:KP}. Section \ref{s3} is devoted to the proof of the existence of ${\mathcal S}^\nu$ weak retractions, that we use in Section \ref{s4} to prove Theorem \ref{main1}.


\section{Preliminaries}\label{s2}

In this section we introduce many concepts and notations needed in the article. We will employ the following conventions: $M\subset\R^m$ and $N\subset\R^n$ are either Nash manifolds or non-singular algebraic sets. Nash subsets of a Nash manifold or algebraic subsets of $\R^n$ are denoted with $X,Y,Z,\ldots$ The semialgebraic sets are denoted with $S,T,R,\ldots$ In addition, ${\mathcal S}^\nu$ and Nash functions on a semialgebraic set are denoted with $f,g,h,\ldots$

Recall some general properties of semialgebraic sets. Semialgebraic sets are closed under Boolean combinations and, by means of quantifier elimination, they are also closed under projections. Any set $S\subset\R^m$ defined by a first order formula in the language of ordered fields is a semialgebraic set \cite[pp.\hspace{2pt}28,\hspace{2pt}29]{bcr}. Thus, the basic topological constructions as the closure of $S$, the interior of $S$ and the boundary of $S$ in $\R^m$ (denoted by $\cl(S)$, $\Int(S)$ and $\partial S$ respectively) are semialgebraic if $S$ is. Also images and preimages of semialgebraic sets by semialgebraic maps are again semialgebraic. The \em dimension \em $\dim(S)$ of a semialgebraic set $S$ is the dimension of its Zariski closure in $\R^m$ \cite[\S2.8]{bcr}. The \em local dimension $\dim(S_x)$ of $S$ at a point $x\in\cl(S)$ \em is the dimension of $U\cap S$ for a small enough open semialgebraic neighborhood $U$ of $x$ in $\R^m$. The dimension of $S$ coincides with the maximum of these local dimensions. For any fixed integer $k$ the set of points $x\in S$ such that $\dim(S_x)=k$ is a semialgebraic subset of $S$. If the function $S \to \N,\ x \mapsto \dim(S_x)$ is constant, then $S$ is said to be \textit{pure dimensional}. 


\subsection{Simplicial approximation and ${\mathcal S}^\nu$ triangulations} \label{subsec:triangulation}

Let $K$ be a finite simplicial complex of $\R^p$. Given any simplex $\sigma\in K$, we denote $\mathrm{Bd}(\sigma)$ the \textit{relative boundary of $\sigma$} defined as the union of proper faces of $\sigma$ and $\sigma^0:=\sigma\,\setminus\, \mathrm{Bd}(\sigma)$ the \textit{relative interior of $\sigma$}, which is equal to the set of points of $\sigma$ whose barycentric coordinates are all strictly positive. The sets $\sigma^0$ are called \textit{open simplexes of $K$}. We indicate $|K|$ the subset $\bigcup_{\sigma\in K}\sigma$ of $\R^p$ equipped with the topology inherited from the Euclidean one of~$\R^p$. Let $K_*$ be the set of vertices of $K$ and for each $v\in K_*$ let $\mathrm{Star}(v,K)$ be the star of $v$ in $K$, that is, the open neighborhood $\bigcup_{\sigma\in K,v\in\sigma}\sigma^0$ of $v$ in~$|K|$. For each positive integer $k$ denote $K^{(k)}$ the $k^{\mathrm{th}}$-iterated barycentric subdi\-vision of $K$. If we fix $\veps>0$ and pick $k$ large enough, then every simplex $\tau\in K^{(k)}$ has diameter $\mathrm{diam}_p(\tau):=\max_{x,y\in\tau}\{\|x-y\|_p\}<\veps$ (see \cite[Thm.15.4]{Mu}).

Let $L$ be a finite simplicial complex of some $\R^q$ and let $g:K_*\to L_*$ be a map between the sets of vertices of $K$ and $L$ satisfying the following condition: \em if $v_1,\ldots,v_r$ are vertices of $K$ that span a simplex of $K$, then $g(v_1),\ldots,g(v_r)$ are vertices of $L$ that span a simplex of $L$\em. Then, $g$ extends uniquely to a continuous map from $|K|$ to $|L|$ whose restriction to each simplex $\sigma$ of $K$ is the restriction to $\sigma$ of an affine map $\R^p\to\R^q$. We denote this extension again $g:|K|\to|L|$ and we say that $g$ is a \em simplicial map\em. Let $f:|K|\to|L|$ be a continuous map. A simplicial map $g:|K|\to|L|$ is called a \em simplicial approximation of~$f$ \em if $f(\mathrm{Star}(v,K))\subset\mathrm{Star}(g(v),L)$ for each vertex $v\in K_*$. If $g$ is a simplicial approximation of $f$, then for each $x\in |K|$ there exists $\xi_x\in L$ such that $\{g(x),f(x)\}\subset\xi_x$, so $\|g(x)-f(x)\|_q\leq\mathrm{diam}_q(\xi_x)$ (see \cite[Cor.14.2]{Mu}). The finite simplicial approximation theorem \cite[Thm.16.1]{Mu} assures that: \em given a continuous map $f:|K|\to|L|$, there exists a positive integer $k$ and a simplicial approximation $g:|K^{(k)}|\to|L|$ of $f$\em. 

\begin{thm}[{\cite[\S 14, 15 \& 16]{Mu}}] \label{s-a}
Let $K$ and $L$ be two finite simplicial complexes and let $f:|K|\to|L|$ be a continuous map. Suppose $|L|\subset\R^q$. Then, for each $\veps>0$, there exist two positive integers $k,\ell$ and a simplicial map $g:|K^{(k)}|\to|L^{(\ell)}|$ such that $\|g(x)-f(x)\|_q<\veps$ for each $x\in|K^{(k)}|=|K|$. 
\end{thm}
\begin{proof}
Choose an integer $\ell\geq1$ such that $\mathrm{diam}_q(\xi)<\veps$ for each $\xi\in L^{(\ell)}$. Now, apply the finite simplicial approximation theorem to the continuous function $f:|K|\to|L^{(\ell)}|=|L|$ and the proof is concluded.
\end{proof}

In the semialgebraic setting we have the following triangulation result, see \cite[Thm.9.2.1 \& Rmk.9.2.3.a)]{bcr}.

\begin{thm}\label{semialg-triangulation}
Given any compact semialgebraic set $S\subset\R^m$, there exist a finite simplicial complex $K$ and a semialgebraic homeomorphism $\Phi:|K|\to S$. In addition, for each $\sigma\in K$ the set $\Phi(\sigma^0)\subset\R^m$ is a Nash manifold and the restriction $\Phi|_{\sigma^0}:\sigma^0\to \Phi(\sigma^0)$ is a Nash diffeomorphism.
\end{thm}

Recently Ohmoto and Shiota proved a remarkable global ${\mathcal S}^1$ version of the latter theorem. We state their result in the compact case only, see \cite[Thm.1.1]{os}.

\begin{thm}\label{C1}
Given any compact semialgebraic set $T\subset\R^n$, there exist a finite simplicial complex $L$ and a semialgebraic homeomorphism $\Psi:|L|\to T$ such that $\Psi\in {\mathcal S}^1(|L|,T)$.
\end{thm}
It is not known if the semialgebraic homeomorphism $\Psi$ can be chosen of class ${\mathcal C}^2$ (see \cite[Sect.1]{os}). However, for a locally ${\mathcal S}^\nu$ polyhedral semialgebraic set we have in addition the following result, see \cite[Prop.I.3.13 \& Rmk.I.3.22]{sh2}.

\begin{thm}\label{Sr-polyhedral}
Let $T\subset\R^n$ be a compact semialgebraic set. If $T$ is locally ${\mathcal S}^\nu$ polyhedral for some integer $\nu\geq1$, then there exist a finite simplicial complex $L$ and a semialgebraic homeomorphism $\Psi:|L|\to T$ such that the restriction of $\Psi$ to $\xi$ belongs to ${\mathcal S}^\nu(\xi,T)$ for each $\xi\in L$.
\end{thm}


\subsection{Spaces of differentiable semialgebraic maps.}\label{smapp}

We will make use in the proof of Theorem \ref{main1} of the following result \cite[Prop.2.D.1]{bfr}. 

\begin{prop}\label{comr}
Let $S\subset\R^m$ and $T\subset\R^n$ be locally compact semialgebraic sets, let $\nu\in\N$ and let $f:T\to T'\subset\R^p$ be an ${\mathcal S}^\nu$ map between semialgebraic sets. Then the map
\[
f_*:{\mathcal S}^\nu(S,T)\to{\mathcal S}^\nu(S,T'),\ g\mapsto f\circ g
\]
is continuous for the ${\mathcal S}^\nu$ topologies.
\end{prop}

The proof of the previous result proposed in \cite{bfr} presents some difficulties concerning the extension of ${\mathcal S}^\nu$ maps whose target is an open subset of an affine space. The proof for the case $\nu=0$, which is the one we need in this paper, is easily arranged and presented next. The proof for the general case is more involved and we include it in Appendix \ref{A}.

\begin{proof}[Proof of Proposition \em \ref{comr} \em for $\nu=0$]
As $S\subset\R^m$ is locally compact it is closed in some open semialgebraic set $U\subset\R^m$. By \cite{dk} we may assume after shrinking $U$ that there exists a ${\mathcal S}^0$ retraction $\tau:U\to S$. Since $T$ is locally compact, it is closed in some open semialgebraic set $V\subset\R^n$ and there exists a semialgebraic ${\mathcal S}^0$ map $F:V\to\R^p$ that extends $f$. We have the following commutative diagram:
\begin{equation*}
\xymatrix{
{\mathcal S}^0(S,T)\ar[r]^{f_*}\ar@{^{(}->}[d]_{\tt j}&{\mathcal S}^0(S,T')\ar@{^{(}->}[d]^{{\tt j}'}\\
{\mathcal S}^0(S,V)\ar[r]^{F_*}&{\mathcal S}^0(S,\R^p)\\
{\mathcal S}^0(U,V)\ar[u]^{\rho_{U}'}\ar[r]^{F_*}&{\mathcal S}^0(U,\R^p)\ar[u]_{\rho_{U}}
}
\end{equation*}
where ${\tt j}$ and ${\tt j}'$ stand for the canonical embeddings and $\rho_U$ and $\rho_U'$ for the restriction homomorphisms. First we explore the lower square. We know that the lower $F_*$ is continuous (it is the Nash manifolds case \cite[II.1.5, p.\hspace{1.5pt}83]{sh}), hence the composition $\rho_U\circ F_*$ is continuous too. The latter map coincides with $F_*\circ\rho_U'$, which is thus continuous. The retraction $\tau:U\to S$ guarantees that $\rho_U'$ is a surjective map. As the target space $V$ is an open semialgebraic subset of an affine space, $\rho_U'$ is an open quotient map by \cite[\S2.C-D]{bfr}, hence the middle $F_*$ is continuous. Now we turn to the upper square. As we have seen that $F_*$ is continuous, the composition $F_*\circ{\tt j}$ is continuous. But this map coincides with ${\tt j}'\circ f_*$, which is consequently continuous. As ${\tt j}'$ is a homeomorphism onto its image, $f_*$ is continuous.
\end{proof}


\subsection{Sets of regular and singular points of a semialgebraic set}\label{s2a}

Let $Z\subset\C^n$ be a complex algebraic set and let $I_\C(Z)$ be the ideal of all polynomials $F\in\C[\x]$ such that $F(z)=0$ for each $z\in Z$. A point $z\in Z$ is \em regular \em if the localization of the polynomial ring $\C[\x]/I_\C(Z)$ at the maximal ideal $\gtM_z$ associated to $z$ is a regular local ring. In this complex setting the Jacobian criterion and Hilbert's Nullstellensatz imply that $z\in Z$ is regular if and only if there exists an open neighborhood $U\subset\C^n$ of $z$ such that $U\cap Z$ is an analytic manifold. We denote $\Reg(Z)$ the set of regular points of $Z$ and it is an open dense subset of $Z$. If $Z$ is irreducible, it is pure dimensional and $\Reg(Z)$ is a connected analytic manifold. In case $Z$ is not irreducible, then the connected components of $\Reg(Z)$ are finitely many analytic manifolds (possibly of different dimensions).

Let $X\subset\R^n$ be a (real) algebraic set and let $I_\R(X)$ be the ideal of all polynomials $f\in\R[\x]$ such that $f(x)=0$ for each $x\in X$. A point $x\in X$ is \em regular \em if the localization of $\R[\x]/I_\R(X)$ at the maximal ideal $\gtm_x$ associated to $x$ is a regular local ring \cite[\S3.3]{bcr}. In addition, $x\in X$ is \em smooth \em if there exists an open neighborhood $U\subset\R^n$ such that $U\cap X$ is a Nash manifold. It holds that each regular point is a smooth point, but in the real case the converse is not always true as it shows the following example.

\begin{exa}\label{snr}
Let $X:=\{(x^2+y^2)xz-y^4=0\} \subset \R^3$. The set of regular points of $X$ is $\Reg(X)=X\setminus\{x=0,y=0\}$. However, the set of smooth points of $X$ is $X\setminus\{0\}$. To prove this fact it suffices to observe that the maps $\varphi_\veps:\{(t,s) \in \R^2: \ t>0\} \to \R^3$ for $\veps=\pm 1$ defined by
\[
\varphi_\veps(s,t):=\veps((s^2+t^2)s^2,(s^2+t^2)st,t^4)
\]
are Nash embeddings, whose images cover the difference $X\setminus\{z=0\}$. It follows that each point $(0,0,a) \in X$ with $a \neq 0$ is smooth. $\sqbullet$
\end{exa}

Let $\widetilde{X}\subset\C^n$ be the complex algebraic set that is the zero set of the extended ideal $I_\R(X)\C[\x]$. We call $\widetilde{X}$ the \em complexification of $X$\em. The ideal $I_\C(\widetilde{X})$ coincides with the ideal $I_\R(X)\otimes_\R\C$, so $\widetilde{X}$ is the smallest complex algebraic subset of $\C^n$ that contains $X$ and
\[
\C[\x]/I_\C(\widetilde{X})\cong(\R[\x]/I_\R(X))\otimes_\R\C.
\]
It holds that the localization $(\R[\x]/I_\R(X))_{\gtm_x}$ is a regular local ring if and only so is its complexification 
\[
(\R[\x]/I_\R(X))_{\gtm_x}\otimes_\R\C\cong(\C[\x]/I_\C(\widetilde{X}))_{\gtM_x}.
\]
Consequently, the \em set of regular points \em of $X$ is $\Reg(X)=\Reg(\widetilde{X})\cap X$ and its \em set of singular points \em is $\Sing(X):=X\setminus\Reg(X)$. The connected components of the open subset $\Reg(X)$ of $X$ is a finite union of Nash manifolds (possibly of different dimensions).

We turn out next to Nash sets. Let $X\subset\R^n$ be a Nash set. A point $x\in X$ is \em regular \em if the localization ${\mathcal N}(X)_{\gtn_x}$ at the maximal ideal $\gtn_x$ of ${\mathcal N}(X)$ associated to $x$ is a regular local ring. Denote $\Reg(X)$ the set of regular points of $X$. Again a point $x\in X$ is \em smooth \em if there exists an open neighborhood $U\subset\R^n$ of $x$ such that $U\cap X$ is a Nash manifold. As before each regular point is a smooth point but Example \ref{snr} shows that the converse is not true in general. The Nash set $X\subset\R^n$ is said to be \em non-singular \em if $X=\Reg(X)$. Assume that $X$ is irreducible. It holds that $X$ is a non-singular Nash set if and only if it is a connected Nash manifold \cite[Def.II.1.12 and Prop.II.5.6]{sh}. Alternatively, this can be shown as an application of Artin-Mazur's Theorem \cite[Thm.8.4.4]{bcr}.

A natural question arises when confronting the definitions of regular point of a real algebraic set $X\subset\R^n$ from the algebraic and Nash viewpoints. Using the properties of completions and henselization \cite[Prop. VII.2.2 and Prop. VII.3.1]{abr} one shows that a point $x\in X$ is regular from the algebraic point of view if and only if it is regular from the Nash point of view.

Note in addition that if the irreducible components of a Nash set $X$ are non-singular, then a point $x\in X$ is regular if and only if it is smooth.


\subsection{Desingularization of algebraic sets}
Let $X\subset Y\subset\R^n$ be algebraic sets such that $Y$ is non-singular. Recall that $X$ is a \em normal-crossings divisor of $Y$ \em if for each point $x\in X$ there exists a regular system of parameters $x_1,\ldots,x_d$ such that $X$ is given in a Zariski neighborhood of $x$ in $Y$ by the equation $x_1\cdots x_k=0$ for some $k=1,\ldots,d$. In particular, the irreducible components of $X$ are non-singular and have codimension $1$ in $Y$. A map $f:=(f_1,\ldots,f_n):Z\to\R^n$ on a (non-empty) subset $Z$ of $\R^m$ is said to be \em regular \em if its components are quotients of polynomials $f_k:=\frac{g_k}{h_k}$ such that $Z\cap\{h_k=0\}=\varnothing$.

The following is a version of Hironaka's desingularization theorems \cite{hi} we will use fruitfully in the sequel.

\begin{thm}[Desingularization]\label{hi1}
Let $X\subset\R^n$ be an algebraic set. Then there exist a non-singular algebraic set $X'\subset\R^m$ and a proper regular map $\phi:X'\to X$ such that $\phi^{-1}(\Sing(X))$ is a normal-crossings divisor of $X'$ and
\[
\phi|_{X'\setminus\phi^{-1}(\Sing(X))}:X'\setminus\phi^{-1}(\Sing(X))\rightarrow X\setminus\Sing X
\]
is a Nash diffeomorphism whose inverse map is regular.
\end{thm}

\begin{rem}
If $X$ is pure dimensional, $X\setminus\Sing X=\Reg(X)$ is dense in $X$. Consequently, as $\phi$ is proper, $\phi$ is also surjective. Furthermore $X'$ is a real algebraic manifold, that is, it is non-singular and pure dimensional. $\sqbullet$
\end{rem}

\section{Proof of Theorems \ref{thm:C1} and \ref{main2}}\label{s2'}

The purpose of this section is to prove Theorems \ref{thm:C1} and \ref{main2}, and Corollary \ref{cor:KP}. We begin developing our tools to approach those proofs.

\subsection{The `shrink-widen' covering and approximation lemmas}

Let $\sigma$ be a simplex of $\R^p$, let $\sigma^0$ be the simplicial interior of $\sigma$ and let $b_\sigma$ be the barycenter of $\sigma$. Given $\veps\in(0,1)$ denote $h_\veps:\R^p\to\R^p,\ x\mapsto b_\sigma+(1-\veps)(x-b_\sigma)$ the homothety of $\R^p$ of center $b_\sigma$ and ratio $1-\veps$. We define the \em $(1-\veps)$-shrinking $\sigma^0_\veps$ of $\sigma^0$ \em by $\sigma^0_\veps:=h_\veps(\sigma^0)$. Observe that $\cl(\sigma^0_\veps)=h_\veps(\sigma)\subset\sigma^0$ for every $\veps\in(0,1)$ and $\sigma^0_\veps$ tends to $\sigma_0$ when $\veps\to 0$. In addition, $\sigma^0=\bigcup_{\veps\in (0,1)}\sigma^0_\veps$ and $\sigma^0_{\veps_2}\subset\sigma^0_{\veps_1}$ if $0<\veps_1\leq\veps_2<1$.

We fix the following notations for the rest of the subsection. Let $S\subset\R^m$ be a compact semialgebraic set, let $K$ be a finite simplicial complex of $\R^p$ and let $\Phi:|K|\to S$ be a semialgebraic homeomorphism such that for each open simplex $\sigma^0$ of $K$ the set $\Phi(\sigma^0)\subset\R^m$ is a Nash manifold and the restriction $\Phi|_{\sigma^0}:\sigma^0\to\Phi(\sigma^0)$ is a Nash diffeomorphism. Define $\pol^0:=\{\Phi(\sigma^0)\}_{\sigma\in K}$ and $\pol:=\{\Phi(\sigma)\}_{\sigma\in K}$. To lighten the notation the elements of $\pol$ will be denoted with the letters $\sfs,\sft,\ldots$ while those of $\pol^0$ with the letters $\sfs^0,\sft^0,\ldots$ in such a way that $\cl(\sfs^0)=\sfs$ and $\sfs^0$ is the interior of $\sfs$ as a semialgebraic manifold-with-boundary. In other words, if $\sfs=\Phi(\sigma)$, then $\sfs^0=\Phi(\sigma^0)$. Moreover, we indicate $\sfs^0_\veps$ the shrinking of $\sfs^0=\Phi(\sigma^0)$ corresponding to $\sigma^0_\veps$ via $\Phi$, that is, $\sfs^0_\veps:=\Phi(\sigma^0_\veps)$.

Consider a Nash tubular neighborhood $\rho_{\sfs^0}:T_{\sfs^0}\to\sfs^0$ of $\sfs^0$ in $\R^m$ and the family of open semialgebraic sets $T_{\sfs^0,\delta}:=\{x\in T_{\sfs^0}:\ \|x-\rho_{\sfs^0}(x)\|_m<\delta\}$ where $\delta>0$. We write $\sfs^0_{\veps,\delta}$ to denote the \em $\delta$-widening of $\sfs^0_\veps$ with respect to $\rho_{\sfs^0}$\em, which is the open neighborhood $\sfs^0_{\veps,\delta}:=(\rho_{\sfs^0})^{-1}(\sfs^0_\veps)\cap T_{\sfs^0,\delta}$ of $\sfs^0_\veps$ in $\R^m$. If $C$ is a closed subset of $\R^m$ such that $C\cap\cl(\sfs^0_\veps)=\varnothing$, there exists $\delta>0$ such that $C\cap\cl(\sfs^0_{\veps,\delta})=\varnothing$. Denote $\rho_{\sfs^0,\veps,\delta}:=\rho_{\sfs^0}|_{\sfs^0_{\veps,\delta}}:\sfs^0_{\veps,\delta}\cap S\to\sfs^0_\veps$ the Nash retraction obtained restricting $\rho_{\sfs^0}$ from $\sfs^0_{\veps,\delta}\cap S$ to $\sfs^0_\veps$.

\begin{lem}\label{lem:covering}
Fix~$\delta>0$. Then, for each $\sfs^0\in\pol^0$ there exist a non-empty open semialgebraic subset $V_{\sfs^0}$ of~$\sfs^0$ (a `shrinking' of $\sfs^0$), an open semialgebraic neighborhood $U_{\sfs^0}$ of $V_{\sfs^0}$ in $S$ (a `widening' of~$V_{\sfs^0}$) satisfying $V_{\sfs^0}=U_{\sfs^0}\cap\sfs^0$ and a Nash retraction $r_{\sfs^0}:U_{\sfs^0}\to V_{\sfs^0}$ such that:
\begin{itemize}
\item[(i)] $\{U_{\sfs^0}\}_{\sfs^0\in\pol^0}$ is an open covering of $S$.
\item[(ii)] $\cl(U_{\sfs^0})\cap\sft=\emptyset$ for each pair $(\sfs^0,\sft)\in\pol^0\times\pol$ satisfying ${\sfs^0}\cap\sft=\emptyset$.
\item[(iii)] $\sup_{x\in U_{\sfs^0}}\{\|x-r_{\sfs^0}(x)\|_m\}<\delta$ for each $\sfs^0\in\pol^0$.
\end{itemize}
\end{lem}
\begin{proof}
Write $d:=\dim(S)$ and $\pol^0_e:=\{\sfs^0\in\pol^0:\, \dim(\sfs^0)\leq e\}$ for $e=0,\ldots,d$. Let us prove by induction on $e\in\{0,1,\ldots,d\}$ that for each $\sfs^0\in\pol^0_e$ there exist an open semialgebraic subset $U_{\sfs^0}^e$ of $S$ and a Nash retraction $r_{\sfs^0}^e:U_{\sfs^0}^e\to V_{\sfs^0}^e:=U_{\sfs^0}^e\cap\sfs^0\neq\emptyset$ such that:
\begin{itemize}
\item[(1)] $\bigcup_{\sfs^0\in\pol^0_e}\sfs^0\subset\bigcup_{\sfs^0\in\pol^0_e}U_{\sfs^0}^e$.
\item[(2)] $\cl(U_{\sfs^0}^e)\cap\sft=\emptyset$ for each pair $(\sfs^0,\sft)\in\pol^0_e\times\pol$ satisfying ${\sfs^0}\cap\sft=\emptyset$.
\item[(3)] $\sup_{x\in U_{\sfs^0}^e}\{\|x-r_{\sfs^0}^e(x)\|_m\}<\delta$ for each $\sfs^0\in\pol^0_e$.
\end{itemize}
Obsviously, the sets $U_{\sfs^0}:=U_{\sfs^0}^d$ and the maps $r_{\sfs^0}:=r_{\sfs^0}^d$ with $\sfs^0\in\pol^0_d=\pol^0$ will be the desired open semialgebraic sets and Nash retractions.

Consider first the case $e=0$. Choose $\delta'\in(0,\delta)$ such that the open ball $B(v,2\delta')$ of~$\R^m$ of center $v$ and radius $2\delta'$ does not meet $\bigcup_{\sft\in\pol,v\not\in\sft}\sft$ for each $\{v\}\in\pol^0_0$. Take~$U_{\{v\}}^0:=B(v,\delta')\cap S$, $V_{\{v\}}^0:=\{v\}$ and $r_{\{v\}}^0:U_{\{v\}}^0\to V_{\{v\}}^0,\ x\mapsto v$ the constant map for each $\{v\}\in\pol^0_0$.

Fix $e\in\{0,\ldots,d-1\}$ and suppose that the assertion is true for such an $e$. As
\[
C_\sigma:=\sigma\setminus\Phi^{-1}\Big(\bigcup_{\tau\in K,\tau\subset\mathrm{Bd}(\sigma)}U^e_{\Phi(\tau^0)}\Big)
\]
is a compact subset of $\sigma^0=\bigcup_{\veps\in (0,1)}\sigma^0_\veps$, there exists $\veps\in (0,1)$ such that $C_\sigma\subset\sigma^0_\veps$ for each $\sigma\in K$ of dimension $e+1$. We have
\[
\textstyle
\bigcup_{\sfs^0\in\pol^0_{e+1}}\sfs^0\subset\bigcup_{\sfs^0\in\pol^0_e}U_{\sfs^0}^e\cup\bigcup_{\sfs^0\in\pol^0_{e+1}\setminus\pol^0_e}\sfs^0_\veps.
\]
If $(\sfs^0,\sft)\in(\pol^0_{e+1}\setminus\pol^0_e)\times\pol$ satisfies $\sfs^0\cap\sft=\emptyset$, then $\cl(\sfs^0_\veps)\cap\sft=\varnothing$ because $\cl(\sfs^0_\veps)\subset\sfs^0$. Pick $\delta'\in(0,\delta)$ such that $\cl(\sfs^0_{\veps,\delta'}\cap S)\cap\sft=\emptyset$ for each pair $(\sfs^0,\sft)\in(\pol^0_{e+1}\setminus\pol^0_e)\times\pol$ satisfying $\sfs^0\cap\sft=\emptyset$. For each $\sfs^0\in\pol^0_{e+1}$ define:
\begin{itemize}
\item $V_{\sfs^0}^{e+1}:=V_{\sfs^0}^e$, $U_{\sfs^0}^{e+1}:=U_{\sfs^0}^e$ and $r_{\sfs^0}^{e+1}:=r_{\sfs^0}^e$ if $\sfs^0\in\pol^0_e$, and \vspace{.3em}
\item $V_{\sfs^0}^{e+1}:=\sfs^0_\veps$, $U_{\sfs^0}^{e+1}:=\sfs^0_{\veps,\delta'}\cap S$ and $r_{\sfs^0}^{e+1}:=\rho_{\sfs^0,\veps,\delta'}$ if $\sfs^0\in\pol^0_{e+1}\setminus\pol^0_e$ (recall that the retraction $\rho_{\sfs^0,\veps,\delta'}$ was defined above the statement of this lemma).
\end{itemize}

The open semialgebraic sets $U_{\sfs^0}^{e+1}$, the non-empty semialgebraic sets $V_{\sfs^0}^{e+1}$ and the retractions $r_{\sfs^0}^{e+1}:U_{\sfs^0}^{e+1}\to V_{\sfs^0}^{e+1}$ for $\sfs^0\in\pol^0_{e+1}$ satisfy conditions $(1)$ to $(3)$, as required. 
\end{proof}

As a consequence of the previous result we obtain the following approximation lemma.

\begin{lem}\label{lem:approximation}
Let $L$ be a finite simplicial complex of $\R^q$, let $g\in {\mathcal S}^0(S,|L|)$ and let $\nu\geq1$ be a positive integer. Suppose that for each $\sft\in\pol$ the restriction $g|_{\sft^0}$ belongs to ${\mathcal S}^\nu(\sft^0,|L|)$ and there exists $\xi_\sft\in L$ such that $g(\sft)\subset\xi_\sft$. Fix $\eta>0$. Then there exists $h\in {\mathcal S}^\nu(S,|L|)$ with the following properties:
\begin{itemize}
\item[(i)] For each $\sft\in\pol$, there exists an open semialgebraic neighborhood $W_\sft$ of $\sft$ in $S$ such that $h(W_\sft)\subset\xi_\sft$.
\item[(ii)] $\|h(x)-g(x)\|_q<\eta$ for each $x\in S$.
\end{itemize} 
\end{lem}
\begin{proof}
As $g$ is uniformly continuous, there exists $\delta>0$ such that 
\begin{equation}\label{eq:uc-g}
\|g(x)-g(x')\|_q<\eta\;\text{ for each pair $x,x'\in S$ satisfying $\|x-x'\|_m<\delta$}.
\end{equation}
By Lemma \ref{lem:covering} for each $\sfs^0\in\pol^0$ there exist an open semialgebraic subset $U_{\sfs^0}$ of $S$ with $V_{\sfs^0}:=U_{\sfs^0}\cap {\sfs^0}\neq\emptyset$ and a Nash retraction $r_{\sfs^0}:U_{\sfs^0}\to V_{\sfs^0}$ such that $\{U_{\sfs^0}\}_{\sfs^0\in\pol^0}$ is a covering of $S$ satisfying:
\begin{align}\label{eq:cl}
&\text{$\cl(U_{\sfs^0})\cap\sft=\emptyset\,$ for each pair $(\sfs^0,\sft)\in\pol^0\times\pol$ satisfying ${\sfs^0}\cap\sft=\emptyset$ and}\\
\label{eq:max}
&\text{$\textstyle\sup_{x\in U_{\sfs^0}}\{\|x-r_{\sfs^0}(x)\|_m\}<\delta\,$ for each $\sfs^0\in\pol^0$.}
\end{align}

Let $\{\theta_{\sfs^0}:S\to[0,1]\}_{\sfs^0\in\pol^0}$ be a ${\mathcal S}^\nu$ partition of unity subordinated to the finite open semialgebraic covering $\{U_{\sfs^0}\}_{\sfs^0\in\pol^0}$ of $S$. For each $\sfs^0\in\pol^0$, the semialgebraic map
\[
g\circ r_{\sfs^0}:U_{\sfs^0}\to V_{\sfs^0}\subset\sfs^0\subset\sfs\to\xi_\sfs,\ x\mapsto r_{\sfs^0}(x)\mapsto g(r_{\sfs^0}(x))
\]
is ${\mathcal S}^\nu$, so also the semialgebraic map $H_{\sfs^0}:S\to\R^q$ defined by 
\begin{equation*}
H_{\sfs^0}(x):=\begin{cases}
\theta_{\sfs^0}(x)\cdot g(r_{\sfs^0}(x))&\text{if $x\in U_{\sfs^0}$,}\\
0&\text{ if $x\in S\setminus U_{\sfs^0}$,}
\end{cases}
\end{equation*}
belongs to ${\mathcal S}^\nu(S,\R^q)$. Consider the ${\mathcal S}^\nu$ map $H:=\sum_{\sfs^0\in\pol^0}H_{\sfs^0}:S\to\R^q$.

Fix $\sft\in\pol$ and define $W_\sft:=S\setminus\bigcup_{\sfs^0\in\pol^0,\,{\sfs^0}\cap\sft=\emptyset}\cl(U_{\sfs^0})$, which is by \eqref{eq:cl} an open semialgebraic neighborhood of $\sft$ in $S$. We claim: $H(W_\sft)\subset\xi_\sft$. 

Pick $x\in W_\sft$. If $\sfs^0\in\pol^0$ and ${\sfs^0}\cap\sft=\emptyset$, then $\theta_{\sfs^0}(x)=0$ because the support of $\theta_{\sfs^0}$ is contained in $U_{\sfs^0}$ and $x\not\in\cl(U_{\sfs^0})$. If ${\sfs^0}\cap\sft\neq\emptyset$, then ${\sfs^0}\subset\sft$, so we conclude
\begin{align}
&\sum_{\substack{\sfs^0\in\pol^0,\,{\sfs^0}\subset\sft,\\x\in U_{\sfs^0}}}\theta_{\sfs^0}(x)=1\ \text{and}\label{eq:sum}\\
H(x)=&\sum_{\substack{\sfs^0\in\pol^0,\,{\sfs^0}\subset\sft,\\x\in U_{\sfs^0}}}\theta_{\sfs^0}(x)g(r_{\sfs^0}(x)).\label{eq:sum2}
\end{align}

If $\sfs^0\in\pol^0$ satisfies ${\sfs^0}\subset\sft$ and $x\in U_{\sfs^0}$, then $r_{\sfs^0}(x)\in V_{\sfs^0}\subset\sfs^0$, so $g(r_{\sfs^0}(x))\in\xi_\sft$. As $\xi_\sft$ is a convex set and each $g(r_{\sfs^0}(x))\in\xi_\sft$ if ${\sfs^0}\subset\sft$ and $x\in U_{\sfs^0}$, we conclude by means of \eqref{eq:sum} and \eqref{eq:sum2} that $H(x)\in\xi_\sft$. Consequently, $H(W_\sft)\subset\xi_\sft$ as claimed. 

As $S=\bigcup_{\sft\in\pol}\sft=\bigcup_{\sft\in\pol}W_\sft$, we deduce $H(S)$ is contained in $|L|$ and $h:S\to|L|,\,x\mapsto H(x)$ is an ${\mathcal S}^\nu$ map that satisfies property (i).

It remains to show that property (ii) holds for $h$. Pick $x\in S$ and observe that $$
\sum_{\sfs^0\in\pol^0,\,x\in U_{\sfs^0}}\theta_{\sfs^0}(x)=1.
$$
Using inequalities \eqref{eq:uc-g} and \eqref{eq:max} we deduce
\begin{align*}
\|h(x)-g(x)\|_q&=\Big\|\sum_{\sfs^0\in\pol^0,\,x\in U_{\sfs^0}}\theta_{\sfs^0}(x)\big(g(r_{\sfs^0}(x))-g(x)\big)\Big\|_q\leq\\
&\leq\sum_{\sfs^0\in\pol^0,\,x\in U_{\sfs^0}}\theta_{\sfs^0}(x)\|g(r_{\sfs^0}(x))-g(x)\|_q<\eta,
\end{align*}
as required.
\end{proof}


\subsection{Proof of Theorem \ref{main2}}
Let $f:S\to T$ be a ${\mathcal S}^0$ map between compact semialgebraic sets $S\subset\R^m$ and $T\subset\R^n$. Suppose $T$ is locally ${\mathcal S}^\nu$~polyhedral for some integer $\nu\geq1$. By Theorem \ref{semialg-triangulation} there exist a finite simplicial complex $K$ and a semialgebraic homeomorphism $\Phi:|K|\to S$ such that the restriction $\Phi|_{\sigma^0}:\sigma^0\to\Phi(\sigma^0)$ is a Nash diffeomorphism for each open simplex $\sigma^0$ of $K$. Define $\pol:=\{\Phi(\sigma)\}_{\sigma\in K}$ and $\pol^0:=\{\Phi(\sigma^0)\}_{\sigma\in K}$. Similarly, Theorem \ref{Sr-polyhedral} implies the existence of a finite simplicial complex $L$ and a semialgebraic homeomorphism $\Psi:|L|\to T$ such that $\Psi|_\xi\in {\mathcal S}^\nu(\xi,T)$ for each $\xi\in L$. Suppose the realization $|L|$ of $L$ belongs to $\R^q$.

Choose an arbitrary $\veps>0$. We will prove \em the existence of a map $H\in {\mathcal S}^\nu(S,T)$ such that $\|H(x)-f(x)\|_n<\veps$ for each $x\in S$\em. 

By the uniform continuity of $\Psi$, there exists $\delta>0$ such that
\begin{equation}\label{eq:uc-Psi}
\|\Psi(z)-\Psi(z')\|_n<\veps\;\text{ for each pair $z,z'\in |L|$ satisfying $\|z-z'\|_q<\delta$}.
\end{equation}
Consider the ${\mathcal S}^0$ map $F:=\Psi^{-1}\circ f\circ\Phi:|K|\to|L|$. By Theorem \ref{s-a} we know that after replacing $K$ and $L$ by suitable iterated barycentric subdivisions there exists a simplicial map $F^*:|K|\to|L|$ such that
\begin{equation}\label{eq:F*}
\|F^*(y)-F(y)\|_q<\delta/2\;\text{ for each $y\in |K|$}.
\end{equation}
Define the ${\mathcal S}^0$ maps $g:=\Psi^{-1}\circ f=F\circ\Phi^{-1}:S\to|L|$ and $g^*:=F^*\circ\Phi^{-1}:S\to|L|$. For each $\sft\in\pol$ the restriction $\Phi^{-1}|_{\sft^0}:\sft^0\to\Phi^{-1}(\sft^0)$ is a Nash diffeomorphism. Thus, as $F^*|_{\Phi^{-1}(\sft^0)}$ is an affine map, $g^*|_{\sft^0} \in {\mathcal S}^\nu(\sft^0,|L|)$. Moreover, there exists $\xi_\sft\in L$ such that $g^*(\sft)\subset\xi_\sft$. By \eqref{eq:F*} we have:
\begin{equation}\label{eq:F**}
\|g^*(x)-g(x)\|_q<\delta/2\;\text{ for each $x\in S$}.
\end{equation}
The following commutative diagram summarizes the situation we have achieved until the moment.
\[
\xymatrix{
S\ar[rr]^f\ar@<0.5ex>[rrdd]^g\ar@<-0.5ex>[rrdd]_{g^*}&&T\\
\\
|K|\ar[uu]^\Phi_\cong\ar@<0.5ex>[rr]^F\ar@<-0.5ex>[rr]_{F^*}&&|L|\ar[uu]_{\Psi}^{\cong}&\ar@{_{(}->}[l]\ar@{_{(}->}[uul]_{\Psi|_\xi}\xi
}
\]
Now, we approximate $g^*$ by a suitable ${\mathcal S}^\nu$ map $h^*$ between $S$ and $|L|$ that will provide, after composing with $\Psi$, the required approximating ${\mathcal S}^\nu$ map $H:=\Psi\circ h^*:S\to T$.

Indeed, by Lemma \ref{lem:approximation} there exist $h^*\in {\mathcal S}^\nu(S,|L|)$ and for each $\sft\in\pol$ an open semialgebraic neighborhood $W_\sft$ of $\sft$ in $S$ satisfying:
\begin{align} 
&\text{$h^*(W_\sft)\subset\xi_\sft$ for each $\sft\in\pol$ and}\label{eq:g*-h*-2}\\
&\text{$\|h^*(x)-g^*(x)\|_q<\delta/2$ for each $x\in S$.}\label{eq:g*-h*-1}
\end{align}
We define $H:=\Psi\circ h^*:S\to T$ and claim: $H\in{\mathcal S}^\nu(S,T)$. 

Recall that $\{W_\sft\}_{\sft\in\pol}$ is an open semialgebraic covering of $S$. Thanks to \eqref{eq:g*-h*-2} the restriction $h^*|_{W_\sft}:W_\sft\to\xi_\sft$ is a well-defined ${\mathcal S}^\nu$ map for each $\sft\in\pol$. In addition, $H|_{W_\sft}=\Psi|_{\xi_\sft}\circ h^*|_{W_\sft}$. As both $\Psi|_{\xi_\sft}$ and $h^*|_{W_\sft}$ are ${\mathcal S}^\nu$ maps, $H|_{W_\sft}$ is also a ${\mathcal S}^\nu$ map. Consequently, $H\in{\mathcal S}^\nu(S,T)$, as claimed.

Next, by \eqref{eq:F**} and \eqref{eq:g*-h*-1} we have
\[
\text{$\|h^*(x)-g(x)\|_q<\delta$ for each $x\in S$}.
\]
Combining the latter inequality with \eqref{eq:uc-Psi}, we conclude
\[
\text{$\|H(x)-f(x)\|_n=\|\Psi(h^*(x))-\Psi(g(x))\|_n<\veps$ for each $x\in S$},
\]
as required.\qed


\subsection{Proof of Theorem \ref{thm:C1}} Let $S$ and $\Phi:|K| \to S$ be as above. Consider a ${\mathcal S}^0$ map $f:S \to \R^n$ and define $T:=f(S)$. By Theorem \ref{C1} there exist a finite simplicial complex $L$ and a semialgebraic homeomorphism $\Psi:|L|\to T$ such that $\Psi\in {\mathcal S}^1(|L|,T)$. Repeating the preceding argument with $\nu=1$, we obtain that for every $\veps>0$ there exists $H \in {\mathcal S}^1(S,T)$ such that $\|H(x)-f(x)\|_n<\veps$.\qed


\subsection{Proof of Corollary \ref{cor:KP}} Let $K$ and $P\subset\R^p$ satisfy the conditions in the statement. Apply Lemma \ref{lem:approximation} to $S:=P$, $\Phi:=\id_P$, $L:=K$, $g:=\id_P$ and $\eta:=2^{-n}$ for each $n\in\N$. We obtain a map $\iota^\nu_n\in{\mathcal S}^\nu(P,P)$ and an open semialgebraic neighborhood $W_\sigma$ of $\sigma$ in $P$ for each $\sigma\in K$ such that: 
\begin{itemize}
\item $\iota^\nu_n(W_\sigma)\subset\sigma$ for each $\sigma\in K$ and
\item $\|x-\iota^{\nu}_n(x)\|_m<2^{-n}$ for each $x\in S$.
\end{itemize}
Thus, the sequence $\{\iota^{\nu}_n\}_n$ converges to the identity map in ${\mathcal S}^0(P,P)$. Consequently, if $f\in {\mathcal S}^0(P)$, the sequence $\{f\circ\iota^\nu_n\}_{n\in\N}$ converges to $f$ in ${\mathcal S}^0(P)$. In addition, if $f|_\sigma\in {\mathcal S}^\nu(\sigma)$ for each $\sigma\in K$, each function $f\circ\iota^\nu_n$ is an ${\mathcal S}^\nu$ function because so is the restriction $(f\circ\iota^\nu_n)|_{W_\sigma}=f|_\sigma\circ\iota^\nu_n|_{W_\sigma}$ for each $\sigma\in K$, as required.\qed

\section{${\mathcal S}^\nu$ weak retractions}\label{s3}

In this section we construct ${\mathcal S}^\nu$ weak retractions $\rho:W\to X$ of open semialgebraic neighborhoods $W$ of a Nash normal-crossings divisor $X$ of a Nash manifold $M$ (Proposition \ref{wr}). Recall that ${\mathcal S}^\nu$ weak retractions $\rho:W\to X$ are ${\mathcal S}^\nu$ maps that are close to the identity on $X$ in the ${\mathcal S}^0$ topology. Their construction requires a preliminary result concerning compatible Nash retractions (Proposition \ref{compatretract}), which is of independent interest. The ${\mathcal S}^\nu$ weak retractions will be used in Section \ref{s4} to prove Theorem \ref{main1}.

\subsection{Compatible Nash retractions}\label{compatretract0}

Let $M\subset\R^m$ be a $d$-dimensional Nash manifold and let $X$ be a Nash subset of $M$. We say that $X$ is a \em Nash normal-crossings divisor of $M$ \em if: 
\begin{itemize}
\item for each point $x\in X$ there exists an open semialgebraic neighborhood $U\subset M$ of $x$ and a Nash diffeomorphism $\varphi:=(\x_1,\ldots,\x_d):U\to\R^d$ such that $\varphi(x)=0$ and $\varphi(X\cap U)=\{\x_1\cdots\x_r=0\}$ for some $r=1,\ldots,d$, and 
\item the (Nash) irreducible components of $X$ are Nash manifolds (of dimension $d-1$). 
\end{itemize}
Assume in the following that $X$ is a Nash normal-crossings divisor of $M$. For each $\ell\geq 2$ define inductively $\Sing_\ell(X):=\Sing(\Sing_{\ell-1}(X))$ and $\Sing_1(X):=\Sing(X)$. The irreducible components of $\Sing_\ell(X)$ are Nash manifolds for each $\ell\geq1$ such that $\Sing_\ell(X)\neq\varnothing$. In fact, if $Y_{\ell,1},\ldots,Y_{\ell,s_\ell}$ are the irreducible components of $\Sing_\ell(X)$, then $\Sing_{\ell+1}(X)=\bigcup_{i\neq j}(Y_{\ell,i}\cap Y_{\ell,j})$. For simplicity we write $\Sing_0(X)=X$ and $\Sing_{-1}(X)=M$.

We claim: \em If $\Sing_\ell(X)\neq\varnothing$, then $\dim(\Sing_\ell(X))=d-\ell-1$\em. 
\begin{proof}
The cases $\ell=-1,0$ are evident. Assume $\ell\geq 1$: As $\dim(X)=d-1$ and $\dim(\Sing_k(X))<\dim(\Sing_{k-1}(X))$, if $\Sing_\ell(X)\neq\varnothing$, then $\dim(\Sing_\ell(X))\leq d-\ell-1$. Pick a point $x\in\Sing_\ell(X)$. Let $U\subset M$ be an open semialgebraic neighborhood of $x$ endowed with a Nash diffeomorphism $\varphi:=(\x_1,\ldots,\x_d):U\to\R^d$ such that $\varphi(x)=0$ and $\varphi(U\cap X)=\{\x_1\cdots\x_\alpha=0\}$ for some $\alpha=1,\ldots,d$. Observe that $\{\x_1=0,\ldots,\x_\alpha=0\}\subset\varphi(\Sing_\ell(X)\cap U)$, so $d-\alpha\leq\dim(\Sing_\ell(X))\leq d-\ell-1$, that is, $\ell+1\leq\alpha$. Now, one proves inductively that $\{\x_1=0,\ldots,\x_{\ell}=0\}$ and $\{\x_2=0,\ldots,\x_{\ell+1}=0\}$ are contained in $\varphi(\Sing_{\ell-1}(X)\cap U)$, so $\{\x_1=0,\ldots,\x_{\ell+1}=0\}\subset\varphi(\Sing_\ell(X)\cap U)$ and $\dim(\Sing_\ell(X))\geq d-\ell-1$.
\end{proof}

We assume that $M$ is irreducible, that is, it is a connected Nash manifold. Let $r\geq0$ be such that $\Sing_r(X)\neq\varnothing$ but $\Sing_{r+1}(X)=\varnothing$. Let $Z$ be an irreducible component of $\Sing_t(X)\neq\varnothing$ for some $0\leq t\leq r$. A Nash retraction $\rho:W\to Z$, where $W\subset M$ is an open semialgebraic neighborhood of $Z$, is \em compatible with $X$ \em if $\rho(X_i\cap W)=X_i\cap Z$ for each irreducible component $X_i$ of $X$ such that $X_i\cap Z\neq\varnothing$.

\begin{prop}[Compatible Nash retractions]\label{compatretract}
There exist an open semialgebraic neighborhood $W\subset M$ of $Z$ and a Nash retraction $\rho:W\to Z$ that is compatible with $X$. In addition, $\rho(Y\cap W)=Y\cap Z$ for each irreducible component $Y$ of $\Sing_\ell(X)$ where $\ell\geq1$.
\end{prop}
\begin{proof}
Fix $\ell\geq0$ such that $\Sing_\ell(X)\neq\varnothing$ and let $Y$ be one of its irreducible components. As $X$ is a Nash normal-crossings divisor of $M$, the intersection $Y\cap Z$ is a Nash manifold. If $Y\cap Z\neq\varnothing$, we consider an open tubular semialgebraic neighborhood $N_Y\subset\R^m$ of $Y\cap Z$ endowed with a Nash retraction $\rho_Y:N_Y\to Y\cap Z$, see \cite[8.9.5]{bcr}. Assume $N_Y\cap N_{Y'}=\varnothing$ if $Y'$ is an irreducible component of $\Sing_\ell(X)$ such that $Y'\cap Z\neq\varnothing$ and $Y\cap Y'\cap Z=\varnothing$. Denote $Y_{\ell,1},\ldots,Y_{\ell,s_\ell}$ the irreducible components of $\Sing_\ell(X)$ for each $-1\leq\ell\leq r$. In particular, $M=\Sing_{-1}(X)=Y_{-1,1}$.

\noindent{\sc Claim: }\em
There exist open semialgebraic neighborhoods 
\[
W_{-1}\subset W_0\subset\cdots\subset W_\ell\subset W_{\ell+1}\subset\cdots\subset W_r
\] 
of $Z$ in $M$ such that $Y_{\ell,j}\cap W_\ell=\varnothing$ if $Y_{\ell,j}\cap Z=\varnothing$ and Nash retractions $\rho_{\ell,j}: Y_{\ell,j}\cap W_\ell\to Y_{\ell,j}\cap Z$ whenever $Y_{\ell,j}\cap Z\neq\varnothing$ satisfying the following compatibility conditions:
\begin{align}
&\rho_{\ell,j_1}|_{Y_{\ell,j_1}\cap Y_{\ell,j_2}\cap W_\ell}=\rho_{\ell,j_2}|_{Y_{\ell,j_1}\cap Y_{\ell,j_2}\cap W_\ell}\quad\text{if $1\leq j_1,j_2\leq s_\ell$},\label{compat}\\
&\rho_{\ell,j}|_{Y_{\ell',k}\cap W_\ell}=\rho_{\ell',k}|_{Y_{\ell',k}\cap W_\ell}\quad\text{for each $\ell',k,j$ such that $Y_{\ell',k}\subset Y_{\ell,j}$}.\label{compat0}
\end{align}
\em

Assume the {\sc Claim} proved for a while. Define $W:=W_{-1}\subset M$, which is an open semialgebraic neighborhood of $Z$, and $\rho:=\rho_{-1,1}:W=Y_{-1,1}\cap W\to Z$, which is a Nash retraction such that $\rho(X_i\cap W)=X_i\cap Z$ for each irreducible component $X_i$ of $X$ satisfying $X_i\cap Z\neq\varnothing$, that is, $\rho$ is compatible with $X$. In addition, $\rho(Y\cap W)=Y\cap Z$ for each irreducible component $Y$ of $\Sing_\ell(X)$ where $\ell\geq1$. Thus, we are reduced to prove the {\sc Claim} above by inverse induction on $\ell$.

\noindent{\sc Step 1:} If $\ell=r$, then $\Sing_{r+1}(X)=\varnothing$, so $\Sing_r(X)$ is a Nash manifold and its irreducible components $Y_{r,j}$ are its connected components. Define $W_r:=M\setminus\bigcup_{Y_{r,j}\cap Z=\varnothing}Y_{r,j}$. We claim: \em if $Y_{r,j}\cap Z\neq\varnothing$, it holds $Y_{r,j}\subset Z$\em. 

Pick a point $x\in Y_{r,j}\cap Z$. Let $U\subset M$ be an open semialgebraic neighborhood of $x$ endowed with a Nash diffeomorphism $\varphi:=(\x_1,\ldots,\x_d):U\to\R^d$ such that $\varphi(x)=0$ and $\varphi(U\cap X)=\{\x_1\cdots\x_\alpha=0\}$ for some $\alpha=1,\ldots,d$. Both $\varphi(U\cap Z)$ and $\varphi(U\cap Y_{r,j})$ are linear coordinate varieties contained in $\{\x_1\cdots\x_\alpha=0\}$ that contain the linear coordinate variety $\{\x_1=0,\ldots,\x_\alpha=0\}$. As $\Sing_{r+1}(X)=\varnothing$, we deduce $\alpha=r+1$ and $\varphi(U\cap Y_{r,j})=\{\x_1=0,\ldots,\x_\alpha=0\}\subset\varphi(U\cap Z)$. Therefore, $U\cap Y_{r,j}\subset U\cap Z$ and by the identity principle $Y_{r,j}\subset Z$, because $Y_{r,j}$ is irreducible. 

Consequently, we define in this case $\rho_{r,j}:=\id_{Y_{r,j}}$.

\noindent{\sc Step 2:} Assume the {\sc Claim} true for $\ell+1,\ldots,r$ and let us check that it is also true for $\ell$.

We have $\Sing(\Sing_\ell(X))=\Sing_{\ell+1}(X)$ and the irreducible components of the set $\Sing_{\ell+1}(X)$ are denoted $Y_{\ell+1,1},\ldots,Y_{\ell+1,s_{\ell+1}}$. By induction hypothesis there exist an open semialgebraic neighborhood $W_{\ell+1}\subset M$ of $Z$ such that $Y_{\ell+1,k}\cap W_{\ell+1}=\varnothing$ if $Y_{\ell+1,k}\cap Z=\varnothing$ and Nash retractions $\rho_{\ell+1,k}:Y_{\ell+1,k}\cap W_{\ell+1}\to Y_{\ell+1,k}\cap Z$ if $Y_{\ell+1,k}\cap Z\neq\varnothing$ satisfying: 
\begin{align}
&\rho_{\ell+1,k_1}|_{Y_{\ell+1,k_1}\cap Y_{\ell+1,k_2}\cap W_{\ell+1}}=\rho_{\ell+1,k_2}|_{Y_{\ell+1,k_1}\cap Y_{\ell+1,k_2}\cap W_{\ell+1}}\quad\text{if $1\leq k_1,k_2\leq s_{\ell+1}$}\label{compat2},\\
&\rho_{\ell+1,k}|_{Y_{\ell',i}\cap W_{\ell+1}}=\rho_{\ell',i}|_{Y_{\ell',i}\cap W_{\ell+1}}\quad\text{for each $\ell',i,k$ such that $Y_{\ell',i}\subset Y_{\ell+1,k}$.}\label{compat20}
\end{align}

Pick $j=1,\ldots,s_\ell$ such that $Y_{\ell,j}\cap Z\neq\varnothing$. If $Y_{\ell,j}\subset Z$, we take $\rho_{\ell,j}:=\id_{Y_{\ell,j}}$ and we are done, so let $J_\ell:=\{j=1,\ldots,s_\ell:\ Y_{\ell,j}\cap Z\neq\varnothing\ \text{and}\ Y_{\ell,j}\not\subset Z\}$ and assume $j\in J_\ell$. 

\label{asbefore} We claim: \em $Y_{\ell,j}\cap Z$ is a Nash manifold of dimension $d-\ell^*-1$ for some $\ell+1\leq\ell^*\leq r$ and it is a union of irreducible components of $\Sing_{\ell^*}(X)\subset\Sing_{\ell+1}(X)$\em. 

As $X$ is a Nash normal-crossings divisor of $M$, the intersection $Y_{\ell,j}\cap Z$ is a Nash manifold. Pick a point $x\in Y_{\ell,j}\cap Z$. Let $U\subset M$ be an open semialgebraic neighborhood of $x$ endowed with a Nash diffeomorphism $\varphi:U\to\R^d$ such that $\varphi(x)=0$ and $\varphi(U\cap X)=\{\x_1\cdots\x_\alpha=0\}$ as above. Both $\varphi(U\cap Z)$ and $\varphi(U\cap Y_{\ell,j})$ are linear coordinate varieties contained in $\{\x_1\cdots\x_\alpha=0\}$ that contain the linear coordinate variety $\{\x_1=0,\ldots,\x_\alpha=0\}$. 

Assume $\varphi(U\cap Z)=\{\x_1=0,\ldots,\x_\beta=0\}$ and $\varphi(U\cap Y_{\ell,j})=\{\x_1=0,\ldots,\x_{\gamma}=0,\x_{\beta+1}=0,\ldots,\x_{\beta-\gamma+\ell+1}=0\}$ for some $\gamma\leq\beta\leq\alpha$ and $\beta-\gamma+\ell+1\leq\alpha$. Define $\ell^*:=\beta-\gamma+\ell$. We have $\gamma<\beta$ because $Y_{\ell,j}\not\subset Z$, so $\ell+1\leq\ell^*$. As 
\[
\varphi(U\cap Y_{\ell,j}\cap Z)=\{\x_1=0,\ldots,\x_{\ell^*+1}=0\},
\]
we deduce $Y_{\ell,j}\cap Z$ is a union of irreducible components of $\Sing_{\ell^*}(X)$, as claimed.

 Let $I_{\ell,j}:=\{i=1,\ldots,s_{\ell^*}:\ Y_{\ell^*,i}\subset Y_{\ell,j}\cap Z\}$ and observe that $Y_{\ell,j}\cap Z=\bigcup_{i\in I_{\ell,j}}Y_{\ell^*,i}$. As $Y_{\ell,j}\cap Z$ is a Nash manifold, the $Y_{\ell^*,i}$ are pairwise disjoint for $i\in I_{\ell,j}$, so the tubular neighborhoods $N_{Y_{\ell^*,i}}$ are pairwise disjoint for $i\in I_{\ell,j}$. Thus, the Nash map
\begin{equation}\label{define}
\rho_{\ell,j}^*:\bigcup_{i\in I_{\ell,j}}N_{Y_{\ell^*,i}}\to Y_{\ell,j}\cap Z=\bigcup_{i\in I_{\ell,j}}Y_{\ell^*,i},\ x\mapsto\rho_{Y_{\ell^*,i}}(x)\quad\text{if $x\in N_{Y_{\ell^*,i}}$} 
\end{equation}
is well-defined and satisfies $\rho_{\ell,j}^*|_{Y_{\ell,j}\cap Z}=\id_{Y_{\ell,j}\cap Z}$, so it is a Nash retraction.

 Define $K_{\ell,j}:=\{k=1,\ldots,s_{\ell+1}:\ Y_{\ell+1,k}\subset Y_{\ell,j}\}$. Observe that $\bigcup_{k\in K_{\ell,j}}Y_{\ell+1,k}$ is a Nash normal-crossings divisor of the Nash manifold $Y_{\ell,j}$. We claim:
\begin{equation} \label{eq:sing-ell}
\Sing_{\ell+1}(X)\cap Y_{\ell,j}=\bigcup_{k\in K_{\ell,j}}Y_{\ell+1,k}.
\end{equation}

Pick a point $x\in\Sing_{\ell+1}(X)\cap Y_{\ell,j}$. Let $U\subset M$ be an open semialgebraic neighborhood of $x$ endowed with a Nash diffeomorphism $\varphi:U\to\R^d$ such that $\varphi(x)=0$ and $\varphi(U\cap X)=\{\x_1\cdots\x_\alpha=0\}$. Both $\varphi(\Sing_{\ell+1}(X)\cap U)$ and $\varphi(U\cap Y_{\ell,j})$ are linear coordinate varieties con\-tained in $\{\x_1\cdots\x_\alpha=0\}$ that contain the linear coordinate variety $\{\x_1=0,\ldots,\x_\alpha=0\}$. As $x\in\Sing_{\ell+1}(X)$, it holds $\alpha\geq\ell+2$. We may assume $\varphi(U\cap Y_{\ell,j})=\{\x_1=0,\ldots,\x_{\ell+1}=0\}$. Observe that $\{\x_1=0,\ldots,\x_{\ell+2}=0\}\subset\varphi(\Sing_{\ell+1}(X)\cap U)$, so there exists $k=1,\ldots,s_{\ell+1}$ such that 
\[
\varphi(Y_{\ell+1,k}\cap U)=\{\x_1=0,\ldots,\x_{\ell+2}=0\}\subset\{\x_1=0,\ldots,\x_{\ell+1}=0\}=\varphi(U\cap Y_{\ell,j}).
\]
As $Y_{\ell+1,k}$ is irreducible, $Y_{\ell+1,k}\subset Y_{\ell,j}$, so $k\in K_{\ell,j}$. Thus, $x\in\bigcup_{k\in K_{\ell,j}}Y_{\ell+1,k}$. The converse inclusion $\bigcup_{k\in K_{\ell,j}}Y_{\ell+1,k}\subset\Sing_{\ell+1}(X)\cap Y_{\ell,j}$ is clear.

 Fix $\nu\geq\dim(M)=d$. Combining \eqref{compat2}, \eqref{eq:sing-ell} and \cite[Thm.1.6 \& Prop.7.6]{bfr}, we deduce the existence of a Nash extension 
\[
f_{\ell,j}:Y_{\ell,j}\cap W_{\ell+1}\to\R^n
\] 
such that $f_{\ell,j}|_{Y_{\ell+1,k}\cap W_{\ell+1}}=\rho_{\ell+1,k}|_{Y_{\ell+1,k}\cap W_{\ell+1}}$ for each $k\in K_{\ell,j}$. 

 As $\rho_{\ell+1,k}|_{Y_{\ell+1,k}\cap Z}=\id_{Y_{\ell+1,k}\cap Z}$ for each $k\in K_{\ell,j}$ and 
\[
Y_{\ell,j}\cap Z=\bigcup_{i\in I_{\ell,j}}Y_{\ell^*,i}\subset\Sing_{\ell+1}(X)\cap Y_{\ell,j}\cap Z=\bigcup_{k\in K_{\ell,j}}(Y_{\ell+1,k}\cap Z),
\]
we deduce $f_{\ell,j}|_{Y_{\ell,j}\cap Z}=\id_{Y_{\ell,j}\cap Z}$. The semialgebraic set $U_{\ell,j}:=f_{\ell,j}^{-1}(\bigcup_{i\in I_{\ell,j}}N_{Y_{\ell^*,i}})$ is an open semialgebraic subset of $Y_{\ell,j}$ that contains $Y_{\ell,j}\cap Z=\bigcup_{i\in I_{\ell,j}}Y_{\ell^*,i}$. 

The semialgebraic set $C_\ell:=\bigcup_{j\in J_\ell}(Y_{\ell,j}\setminus U_{\ell,j})$ is a closed semialgebraic subset of $M$ and $Z\cap C_\ell=\varnothing$, because $Y_{\ell,j}\cap Z\subset U_{\ell,j}$ for each $j\in J_\ell$ and
\[
Z\cap C_\ell=\bigcup_{j\in J_\ell}((Y_{\ell,j}\cap Z)\setminus U_{\ell,j})=\varnothing.
\]
Define 
\[
W_\ell:=W_{\ell+1}\setminus\Big(C_\ell\cup\bigcup_{Y_{\ell,k}\cap Z=\varnothing}Y_{\ell,k}\Big)\subset W_{\ell+1}.
\] 
Observe that $Z\subset W_\ell$ and $Y_{\ell,j}\cap W_\ell\subset Y_{\ell,j}\setminus C_\ell\subset Y_{\ell,j}\setminus(Y_{\ell,j}\setminus U_{\ell,j})\subset U_{\ell,j}$ for each $j\in J_\ell$. 

Consider the composition
\[
\rho_{\ell,j}:=\rho_{\ell,j}^*\circ f_{\ell,j}|_{Y_{\ell,j}\cap W_\ell}:Y_{\ell,j}\cap W_\ell\subset U_{\ell,j}\overset{f_{\ell,j}|_{U_{\ell,j}}}{\longrightarrow}\bigcup_{i\in I_{\ell,j}}N_{Y_{\ell^*,i}}\overset{\rho_{\ell,j}^*}{\longrightarrow} Y_{\ell,j}\cap Z,
\] 
where $\rho_{\ell,j}^*$ is the Nash retraction defined in \eqref{define} for each $j\in J_\ell$. 

 Note that $\rho_{\ell,j}:Y_{\ell,j}\cap W_\ell\to Y_{\ell,j}\cap Z$ is a Nash retraction such that $\rho_{\ell,j}|_{Y_{\ell+1,k}\cap W_\ell}=\rho_{\ell+1,k}|_{Y_{\ell+1,k}\cap W_\ell}$ for each $k\in K_{\ell,j}$, because 
\begin{multline*}
f_{\ell,j}|_{Y_{\ell+1,k}\cap W_\ell}=\rho_{\ell+1,k}|_{Y_{\ell+1,k}\cap W_\ell},\quad f_{\ell,j}(Y_{\ell+1,k}\cap W_\ell)=Y_{\ell+1,k}\cap Z\\
\text{and}\quad\rho_{\ell,j}^*|_{Y_{\ell,j}\cap Z}=\id_{Y_{\ell,j}\cap Z}. 
\end{multline*}

Let $\ell',i,j$ be such that $Y_{\ell',i}\subset Y_{\ell,j}$ and $\ell'\geq\ell+1$. Then
\[
Y_{\ell',i}\subset\Sing_{\ell'}(X)\cap Y_{\ell,j}\subset\Sing_{\ell+1}(X)\cap Y_{\ell,j}=\bigcup_{k\in K_{\ell,j}}Y_{\ell+1,k}.
\]
As $Y_{\ell',i}$ is irreducible, there exists $k\in K_{\ell,j}$ such that $Y_{\ell',i}\subset Y_{\ell+1,k}$. We have by \eqref{compat20}
\[
\rho_{\ell,j}|_{Y_{\ell',i}\cap W_\ell}=(\rho_{\ell,j}|_{Y_{\ell+1,k}\cap W_\ell})|_{Y_{\ell',i}\cap W_\ell}=(\rho_{\ell+1,k}|_{Y_{\ell+1,k}\cap W_\ell})|_{Y_{\ell',i}\cap W_\ell}=\rho_{\ell',i}|_{Y_{\ell',i}\cap W_\ell}.
\]
If $\ell'=\ell$ and $Y_{\ell,i}\subset Y_{\ell,j}$, we deduce $Y_{\ell,i}=Y_{\ell,j}$, that is, $i=j$ and there is nothing to prove. Consequently, the Nash retractions $\rho_{\ell,j}$ satisfy condition \eqref{compat0}.

 Let $1\leq j_1,j_2\leq s_\ell$ be such that $Y_{\ell,j_1}\cap Y_{\ell,j_2}\cap W_\ell\neq\varnothing$. Proceeding analogously to \ref{asbefore}, one shows that the intersection $Y_{\ell,j_1}\cap Y_{\ell,j_2}$ is a Nash manifold whose connected components are (pairwise disjoint) irreducible components $Y_{\ell',i}$ of $\Sing_{\ell'}(X)$ for some $\ell'\geq\ell+1$, which are obviously contained in $Y_{\ell,j_k}$ for $k=1,2$. As $\rho_{\ell,j_k}|_{Y_{\ell',i}\cap W_\ell}=\rho_{\ell',i}|_{Y_{\ell',i}\cap W_\ell}$ for each $\ell',i,\ell$ such that $Y_{\ell',i}\subset Y_{\ell,j_k}$, we conclude
\begin{equation*}
\rho_{\ell,j_1}|_{Y_{\ell,j_1}\cap Y_{\ell,j_2}\cap W_\ell}=\rho_{\ell,j_2}|_{Y_{\ell,j_1}\cap Y_{\ell,j_2}\cap W_\ell},
\end{equation*}
so the Nash retractions $\rho_{\ell,j}$ satisfy condition \eqref{compat}. This finishes the induction step and we are done.
\end{proof}

\subsection{${\mathcal S}^\nu$ weak retractions}
Fix a Nash manifold $M\subset\R^m$ of dimension $d$, a Nash normal-crossings divisor $X$ of $M$ and an integer $\nu\geq1$. The purpose of this subsection is to prove the existence of ${\mathcal S}^\nu$ weak retractions, that is, ${\mathcal S}^\nu$ maps $\rho:W\to X$, where $W\subset M$ is an open semialgebraic neighborhood of $X$, that are ${\mathcal S}^0$ close to the identity on $X$. Namely,

\begin{prop}[${\mathcal S}^\nu$ weak retractions]\label{wr}
There exist ${\mathcal S}^\nu$ weak retractions $\rho:W\to X$ that are ${\mathcal S}^0$ arbitrarily close to the identity on $X$. More explicitly, there exists an open semialgebraic neighborhood $W \subset M$ of $X$ with the following property: for each neighborhood $\mathcal U \subset {\mathcal S}^0(X,X)$ of $\id_X$, there exists a map $\rho \in {\mathcal S}^\nu(W,X)$ such that $\rho|_X \in {\mathcal U}$.
\end{prop}

Before proving this we need a preliminary result.

\begin{lem}[${\mathcal S}^\nu$ double collar]\label{neighcollar}
Let $Y\subset M$ be a Nash submanifold of dimension $d-1$ that is closed in $M$. Let $V\subset M$ be an open semialgebraic neighborhood of $Y$ and $\pi:V\to Y$ a ${\mathcal S}^\nu$ retraction. Let $h:V\to\R$ be a ${\mathcal S}^\nu$ function such that $Y\subset\{h=0\}$ and $d_xh:T_xM\to\R$ is surjective for each $x\in Y$. Consider the ${\mathcal S}^\nu$ map $\varphi:=(\pi,h):V\to Y\times\R$. Then there exist an open semialgebraic neighborhood $W\subset V$ of $Y$ and a strictly positive ${\mathcal S}^\nu$ function $\veps:Y\to\R$ such that $\varphi(W)=\{(x,t)\in Y\times\R:\ |t|<\veps(x)\}$ and $\varphi|_W:W\to\varphi(W)$ is a ${\mathcal S}^\nu$ diffeomorphism.
\end{lem}
\begin{proof}
We show first: \em The derivative $d_x\varphi=(d_x\pi,d_xh):T_xM\to T_xY\times\R$ is an isomorphism for each $x\in Y$\em. As $\dim(T_xM)=\dim(T_xY\times\R)$, it is enough to show: \em $d_x\varphi$ is surjective\em. 

As $\varphi|_{Y}=(\id_{Y},0)$, we have $d_x\varphi|_{T_xY}=(\id_{T_xY},0)$, so $T_xY\times\{0\}\subset\im(d_x\varphi)$. In addition $d_xh:T_xM\to\R$ is surjective, so there exists $v\in T_xM$ such that $d_xh(v)=1$. Hence, $d_x\varphi(v)=(d_x\pi(v),1)$ and $d_x\varphi$ is surjective.

Let $V':=\{x\in V:\ d_x\varphi\ \text{is an isomorphism}\}$, which is an open semialgebraic neighborhood of $Y$ in $V$. Thus, $\varphi|_{V'}:V'\to Y\times\R$ is an open map and $\varphi(V')$ is an open semialgebraic neighborhood of $Y\times\{0\}$ in $Y\times\R$. As $\varphi|_{V'}:V'\to\varphi(V')$ is a local semialgebraic homeomorphism and $\varphi|_{Y}=(\id_{Y},0)$ is a semialgebraic homeomorphism (onto its image), there exist by \cite[Lem.9.2]{bfr} open semialgebraic neighborhoods $V''\subset V'$ of $Y$ and $U\subset Y\times\R$ of $Y\times\{0\}$ such that $\varphi|_{V''}:V''\to U$ is a semialgebraic homeomorphism.

Consider the strictly positive continuous semialgebraic function
\[
\delta:Y\to(0,+\infty),\ x\mapsto\dist((x,0),(Y\times\R)\setminus U).
\] 
By \cite[II.4.1, p.\hspace{2pt}123]{sh} there exists a strictly positive Nash function $\veps$ on $Y$ such that $\frac{1}{2}\delta<\veps<\delta$. Consider the open semialgebraic neighborhood $U':=\{(x,t)\in Y\times\R:\ |t|<\veps(x)\}\subset U$ of $Y\times\R$ and define $W:=(\varphi|_{V''})^{-1}(U')$. The restriction $\varphi|_W:W\to U'$ is a ${\mathcal S}^\nu$ diffeomorphism, as required.
\end{proof}

\begin{proof}[Proof of Proposition \em\ref{wr}]\setcounter{paragraph}{0}
Let $X_1,\ldots,X_s$ be the irreducible components of $X$ and fix $j=1,\ldots,s$. By \cite[II.6.2]{sh} there exists a Nash vector subbundle $(\mathscr{E}_j,\theta_j,X_j)$ of the trivial Nash vector bundle $(X_j\times\R^m,\vartheta_j,X_j)$, a (strictly) positive Nash function $\delta_j$ on $X_j$ and a Nash diffeomorphism $\chi_j:V_j\to\mathscr{E}_{j,\delta_j}$ from a semialgebraic open neighborhood $V_j\subset M$ of $X_j$ onto 
\[
\mathscr{E}_{j,\delta_j}:=\{(x,y)\in\mathscr{E}_j:\ \|y\|_m<\delta_j(x)\}
\]
such that $\chi_j|_{X_j}=(\id_{X_j},0)$. The tuple $(V_j,\chi_j,\mathscr{E}_j,\theta_j,X_j,\delta_j)$ is a Nash tubular neighborhood of $X_j$ in $M$ and the composition $\Theta_j:={\theta_j}|_{\mathscr{E}_{j,\delta_j}}\circ\chi_j:V_j\to X_j$ is a Nash retraction. As $X_j$ is a Nash hypersurface of $M$, the Nash subbundle $(\mathscr{E}_j,\theta_j,X_j)$ has rank $1$, that is, it is a line bundle. Shrinking $V_j$ if necessary, we may assume in addition $V_j\cap X_k=\varnothing$ whenever $X_j\cap X_k=\varnothing$. We refer the reader to \cite[\S III.10]{o} for the orientability of vector bundles. We distinguish two cases. 

\noindent{\sc Case 1.} \em Assume first that $\mathscr{E}_j$ is a trivial Nash line bundle \em (or equivalently, it is an orientable Nash line bundle \cite[Def.III.10.4]{o}). Then, $(\mathscr{E}_j,\theta_j,X_j)$ is a Nash diffeomorphic to the trivial bundle $(X_j\times\R,\pi_j,X_j)$. This means that there exists a Nash diffeomorphism 
\[
\mu_j:=(\lambda_j,h_j):V_j\to\{(x,y)\in X_j\times\R:\ |y|<\delta_j(x)\},\ z\mapsto(\lambda_j(z),h_j(z))
\]
such that $\mu_j|_{X_j}=(\id_{X_j},0)$. Thus, $X_j=h_j^{-1}(0)$, $d_xh_j:T_xM\to\R$ is surjective for each $x\in X_j$ and $\ker(d_xh_j)=T_xX_j$. Consequently, if $x\in X_j \cap X_k$ for some $k\neq j$, then $d_x(h_j|_{X_k}):T_xX_k\to\R$ is also surjective because $\ker(d_xh_j)=T_xX_j$ and $X_j$ and $X_k$ are transverse at $x$ in $M$.
 
 By Proposition \ref{compatretract} there exist an open semialgebraic neighborhood $W_j\subset V_j$ of $X_j$ together with a Nash retraction $\rho_j:W_j\to X_j$ compatible with $X$. Observe that $W_j\cap X_k\subset V_j\cap X_k=\varnothing$ whenever $X_j\cap X_k=\varnothing$. After shrinking $W_j$ there exists by Lemma \ref{neighcollar} a strictly positive ${\mathcal S}^\nu$ function $\veps_j:X_j\to\R$ such that the ${\mathcal S}^\nu$ map 
\[
\phi_j:=(\rho_j,h_j):W_j\to\Omega_j:=\{(x,t)\in X_j\times\R:\ |t|<\veps_j(x)\}
\]
is an ${\mathcal S}^\nu$ diffeomorphism. As $\rho_j(W_j\cap X_k)=X_j\cap X_k$, we can also assume by Lemma \ref{neighcollar} $\phi_j(W_j\cap X_k)=\Omega_j\cap((X_j\cap X_k)\times\R)$ for each $k=1,\ldots,s$ (recall that $d_x(h_j|_{X_k}):T_xX_k\to\R$ is surjective at each $x\in X_k\cap X_j$ for $k\neq j$). Let $\eta_j:X_j\to\R$ be a strictly positive ${\mathcal S}^\nu$ function such that $\eta_j<\veps_j$ on $X_j$ and let $f:\R\to[0,1]$ be a ${\mathcal S}^\nu$ function such that $f(t)=0$ for $|t|\leq\frac{1}{3}$ and $f(t)=1$ for $|t|\geq\frac{1}{2}$. Consider the ${\mathcal S}^\nu$ map 
\[
\varphi_j:\Omega_j\to\Omega_j,\ (x,t)\mapsto(x,f(t/\eta_j(x))t).
\]
Define $\psi_j:=(\phi_j^{-1}\circ\varphi_j\circ\phi_j):W_j\to W_j$ and observe that $\psi_j(W_j\cap X_k)\subset W_j\cap X_k$ for each $k=1,\ldots,s$. We extend $\psi_j$ by the identity to the whole $M$ and obtain a ${\mathcal S}^\nu$ map $\Psi_j:M\to M$ such that:
\begin{itemize}
\item $\Psi_j(W^*_j)=X_j$ for $W_j^*:=\phi_j^{-1}(\{(x,t)\in X_j\times\R:\ |t|\leq\eta_j(x)/3\})$.
\item $\Psi_j(y)=y$ if $y\in M\setminus\phi_j^{-1}(\{(x,t)\in X_j\times\R:\ |t|<\eta_j(x)/2\})$.
\item $\Psi_j(X_k)\subset X_k$ for each $k=1,\ldots,s$.
\item $\Psi_j$ is arbitrarily close to the identity on $X$ if $\eta_j$ is small enough.
\end{itemize}
Only the last assertion requires a further comment. As $\Psi_j$ is the identity on the difference $ M\setminus\phi_j^{-1}(\{(x,t)\in X_j\times\R:\ |t|<\eta_j(x)/2\})$ and $\phi_j$ is a Nash diffeomorphism (in particular a proper map), it is enough to prove by \cite[Rem.II.1.5]{sh} that $\varphi_j$ is close to the identity on $U:=\{(x,t)\in X_j\times\R:\ |t|<\eta_j(x)/2\}$. If $(x,t)\in U$, we have $|f(t/\eta_j(x))-1|\leq 1$ and
\[
\|\varphi_j(x,t)-(x,t)\|_{m+1}=|f(t/\eta_j(x))-1||t|\leq|t|<\eta_j(x)/2<\eta_j(x),
\]
so $\Psi_j$ is arbitrarily close to the identity on $X$ if $\eta_j$ is small enough.

\noindent{\sc Case 2.} \em Assume next that $\mathscr{E}_j$ is not orientable\em. Let $\pi_j:\widetilde{X}_j\to X_j$ be a Nash double cover such that the pull-back $(\pi_j^*\mathscr{E}_j,\theta_j',\widetilde{X}_j)$ is an orientable Nash line bundle \cite[Cor.III.10.6]{o}, hence $\pi_j^*\mathscr{E}_j\cong\widetilde{X}_j\times\R$ is a trivial Nash line bundle over $\widetilde{X}_j$. We can take $\widetilde{X}_j:=\{(x,y)\in\mathscr{E}_j:\ \|y\|_m=1\}$ (which is the unit sphere bundle in $\mathscr{E}_j$ with respect to the metric induced by that of $X_j\times\R^m$) and $\pi_j:=\theta_j|_{\widetilde{X}_j}:\widetilde{X}_j\to X_j$. Consider the Nash morphism of Nash line bundles $\gamma_j:\pi^*_j\mathscr{E}_j\to\mathscr{E}_j$ that makes the following diagram commutative
\[
\xymatrix{
\pi^*_j\mathscr{E}_j\ar[d]_{\theta_j'}\ar[r]^{\gamma_j}&\mathscr{E}_j\ar[d]^{\theta_j}&\ar@{_{(}->}[l]\mathscr{E}_{j,\delta_j}&\ar[l]_(0.4){\chi_j}^\cong\ar@<0.7ex>[lld]^{\Theta_j}V_j\\
\widetilde{X}_j\ar[r]^{\pi_j}&X_j\ar@{^{(}->}@<0.1ex>[rru]
}
\]
Notice that $\gamma_j:\pi^*_j\mathscr{E}_j\to\mathscr{E}_j$ is a Nash doble cover. Denote $\widetilde{V}_j:=\gamma_j^{-1}(\mathscr{E}_{j,\delta_j})$, consider the Nash double cover $\widetilde{\pi}_j:=\chi_j^{-1}\circ\gamma_j|_{\widetilde{V}_j}:\widetilde{V}_j\to V_j$ and identify $\widetilde{X}_j$ with $\widetilde{X}_j\times\{0\}\subset\pi^*_j\mathscr{E}_j$. Define $\widetilde{X}:=\widetilde{\pi}_j^{-1}(X\cap V_j)$, which is a Nash normal-crossings divisor of the Nash manifold $\widetilde{V}_j$, and $\widetilde{X}_k:=\widetilde{\pi}_j^{-1}(X_k\cap V_j)$ for $k=1,\ldots,s$. Each $\widetilde{X}_k$ is a finite union of disjoint irreducible components of $\widetilde{X}$. Denote the regular involution that generate the $\Z_2$ deck transformation group of the Nash double cover $\gamma_j:\pi^*_j\mathscr{E}_j\to\mathscr{E}_j$ with $\sigma_j:\pi^*_j\mathscr{E}_j\to\pi^*_j\mathscr{E}_j$. Recall that $\sigma_j$ has no fixed points, it inverts the orientation and $\gamma_j\circ\sigma_j=\gamma_j$. Consequently, the same happens with $\sigma_j|_{\widetilde{V}_j}:\widetilde{V}_j\to\widetilde{V}_j$, that is, it has no fixed points, it inverts the orientation and $\widetilde{\pi}_j\circ\sigma_j|_{\widetilde{V}_j}=\widetilde{\pi}_j$. For simplicity we denote $\sigma_j|_{\widetilde{V}_j}$ with $\sigma_j$.

 By Proposition \ref{compatretract} there exists a Nash retraction $\rho_j:W_j\to X_j$ compatible with $X$ where $W_j\subset V_j$ is an open semialgebraic neighborhood of $X_j$. Define $\widetilde{W}_j:=\widetilde{\pi}_j^{-1}(W_j)$, which is invariant under $\sigma_j$. We claim: \em The Nash retraction $\rho_j:W_j\to X_j$ lifts to a Nash retraction $\widetilde{\rho}_j:\widetilde{W}'_j\to\widetilde{X}_j$ compatible with $\widetilde{X}$ for some open semialgebraic neighborhood $\widetilde{W}'_j\subset\widetilde{W}_j$ of $\widetilde{X}_j$\em.

As $\widetilde{\pi}_j|_{\widetilde{W}_j}:\widetilde{W}_j\to W_j$ is a local Nash diffeomorphism, there exists by \cite[9.3.9]{bcr} a finite open semialgebraic covering $\widetilde{W}_j=\bigcup_{k=1}^\ell A_k$ such that $\widetilde{\pi}_j|_{A_k}:A_k\to B_k:=\widetilde{\pi}_j(A_k)$ is a Nash diffeomorphism. If we consider the covering $\{A_k,\sigma_j(A_k):\ k=1,\ldots,\ell\}$ of $\widetilde{W}_j$ we may assume in addition $\sigma_j(A_k)=A_{i(k)}$, $B_k=B_{i(k)}$ and $\widetilde{\pi}_j|_{A_k}=\widetilde{\pi}_j|_{A_{i(k)}} \circ \sigma_j|_{A_k}$ for each $k=1,\ldots,\ell$ and for some $i(k)=1,\ldots,\ell$. Observe that $W_j=\bigcup_{k=1}^\ell B_k$ is a finite open semialgebraic covering. Let $E_k:=B_k\cap X_j$ and observe that $X_j=\bigcup_{k=1}^\ell E_k$ is an open semialgebraic covering. Define $B_k':=B_k\cap\rho_j^{-1}(E_k)$, $W_j':=\bigcup_{k=1}^\ell B_k'$, $A_k':=A_k\cap\widetilde{\pi}_j^{-1}(B_k')$, $\widetilde{W}_j':=\bigcup_{k=1}^\ell A_k'$ and $D_k:=A_k'\cap\widetilde{X}_j$. Observe that $\widetilde{X}_j=\bigcup_{k=1}^\ell D_k$, the restriction map $\widetilde{\pi}_j|_{D_k}:D_k\to E_k$ is a Nash diffeomorphism, $E_k=B_k'\cap X_j$, $\rho_j|_{B_k'}:B_k'\to E_k$ is a Nash retraction, $E_k=E_{i(k)}$, $B'_k=B'_{i(k)}$, $\sigma_j(A'_k)=A'_{i(k)}$ and $\sigma_j(\widetilde{W}'_j)=\widetilde{W}'_j$. Consider the commutative diagram
\[
\xymatrix{
&\widetilde{W}_j\ar[ddl]_{\widetilde{\pi}_j|_{\widetilde{W}_j}}&\widetilde{W}_j'\ar[r]^{\widetilde{\rho}_j}\ar@{_{(}->}[l]&\widetilde{X}_j\\
&A_k\ar@{^{(}->}[u]\ar[d]_{\widetilde{\pi}_j|_{A_k}}&\ar@{_{(}->}[l]A_k'\ar@{^{(}->}[u]\ar[r]^{\widetilde{\rho}_j|_{A_k'}}\ar[d]_{\widetilde{\pi}_j|_{A_k'}}&D_k\ar@{^{(}->}[u]\ar[d]^{\widetilde{\pi}_j|_{D_k}}\\
W_j&\ar@{_{(}->}[l]B_k&\ar@{_{(}->}[l]B_k'\ar[r]^{\rho_j|_{B_k'}}&E_k
}
\]
where 
\[
\widetilde{\rho}_j:\widetilde{W}_j'\to\widetilde{X}_j,\ x\mapsto(\widetilde{\pi}_j|_{D_k})^{-1}((\rho_j\circ\widetilde{\pi}_j)(x))\quad\text{if $x\in A_k'$.}
\]

 \em The map $\widetilde{\rho}_j$ is well-defined, it is a Nash retraction \em (because $\rho_j$ is a Nash retraction) \em and $\rho_j\circ\widetilde{\pi}_j=\widetilde{\pi}_j\circ\widetilde{\rho}_j$ on $\widetilde{W}'_j$\em. In addition, \em $\widetilde{\rho}_j$ is compatible with $\widetilde{X}$ \em (because $\rho_j$ is compatible with $X$) \em and $\widetilde{\rho}_j\circ\sigma_j=\sigma_j\circ\widetilde{\rho}_j$ on $\widetilde{W}'_j$\em.

To prove that $\widetilde{\rho}_j$ is well-defined pick $x\in A_k'\cap A_i'$. Then $\widetilde{\pi}_j(x)\in\widetilde{\pi}_j(A_k')\cap\widetilde{\pi}_j(A_i')=B_k'\cap B_i'$, so $\rho_j(\widetilde{\pi}_j(x))\in E_k\cap E_i=B_k'\cap B_i'\cap X_j$. As $D_k\cap D_i=A_k'\cap A_i'\cap\widetilde{X}_j$ and $\widetilde{\pi}_j|_{D_k\cap D_i}:D_k\cap D_i\to E_k\cap E_i$ is a Nash diffeomorphism, there exists a unique $y\in D_k\cap D_i$ such that $(\rho_j\circ\widetilde{\pi}_j)(x)=\widetilde{\pi}_j(y)$, so $(\widetilde{\pi}_j|_{D_k})^{-1}((\rho_j\circ\widetilde{\pi}_j)(x))=(\widetilde{\pi}_j|_{D_i})^{-1}((\rho_j\circ\widetilde{\pi}_j)(x))$ and $\widetilde{\rho}_j$ is well-defined.

Let us check that $\widetilde{\rho}_j\circ\sigma_j=\sigma_j\circ\widetilde{\rho}_j$ on $\widetilde{W}'_j$. It is enough to prove this property locally. As $\widetilde{\pi}_j\circ\sigma_j=\widetilde{\pi}_j$ and $\sigma_j$ is an involution, we have $\widetilde{\pi}_j|_{D_k}\circ\sigma_j|_{D_{i(k)}}=\widetilde{\pi}_j|_{D_{i(k)}}$ and $\sigma_j|_{D_{i(k)}}^{-1}=\sigma_j|_{D_k}$, so $\sigma_j|_{D_k}\circ(\widetilde{\pi}_j|_{D_k})^{-1}=(\widetilde{\pi}_j|_{D_{i(k)}})^{-1}$. Thus,
\begin{align*}
(\sigma_j\circ\widetilde{\rho}_j)|_{A'_k}&=\sigma_j|_{D_k}\circ\big((\widetilde{\pi}_j|_{D_k})^{-1}\circ(\rho_j\circ\widetilde{\pi}_j)|_{A'_k}\big)\\
&=(\widetilde{\pi}_j|_{D_{i(k)}})^{-1} \circ \rho_j|_{B'_k} \circ \widetilde{\pi}_j|_{A'_k}\\
&=(\widetilde{\pi}_j|_{D_{i(k)}})^{-1} \circ \rho_j|_{B'_{i(k)}} \circ \big(\widetilde{\pi}_j|_{A'_{i(k)}} \circ \sigma_j|_{A'_k}\big)\\
&=\big((\widetilde{\pi}_j|_{D_{i(k)}})^{-1}\circ(\rho_j\circ\widetilde{\pi}_j)|_{A'_{i(k)}}\big)\circ\sigma_j|_{A'_k}=(\widetilde{\rho}_j\circ\sigma_j)|_{A'_k}
\end{align*}
for each $k$. The fact that $\widetilde{\rho}_j$ is compatible with $\widetilde{X}$ is clear by construction.

 After shrinking $\widetilde{W}_j'$ we may assume (as in {\sc Case 1}) that there exists an ${\mathcal S}^\nu$ function $\widetilde{h}_j:\widetilde{V}_j\to\R$ such that: 
\begin{itemize}
\item[(i)] $\widetilde{X}_j\subset\{\widetilde{h}_j=0\}$. 
\item[(ii)] $d_x\widetilde{h}_j:T_x\widetilde{V}_j\to\R$ is surjective for each $x\in\widetilde{X}_j$. 
\end{itemize}

Substituting $\widetilde{h}_j$ by $\widetilde{h}_j':=\widetilde{h}_j-\widetilde{h}_j\circ\sigma_j$, we may assume in addition $\widetilde{h}_j\circ\sigma_j=-\widetilde{h}_j$. Let us check that such change keeps properties (i) and (ii). As $\sigma_j(\widetilde{X}_j)=\widetilde{X}_j$, we have $\widetilde{X}_j\subset\{\widetilde{h}_j'=0\}$ (so property (i) holds for $\widetilde{h}_j'$). Pick $x\in\widetilde{X}_j$. The isomorphism $d_x\sigma_j:T_x\widetilde{V}_j\to T_{\sigma_j(x)}\widetilde{V}_j$ inverts the orientation and in fact if $v\in T_x\widetilde{V}_j\setminus T_x\widetilde{X}_j$ satisfies $d_x\widetilde{h}_j(v)>0$, then $d_x\sigma_j(v)\in T_{\sigma_j(x)}\widetilde{V}_j$ satisfies $d_{\sigma_j(x)}\widetilde{h}_j(d_x\sigma_j(v))<0$. Consequently, 
\[
d_x\widetilde{h}_j'(v)=d_x\widetilde{h}_j(v)-d_{\sigma_j(x)}\widetilde{h}_j(d_x\sigma_j(v))>0
\]
and $d_x\widetilde{h}_j':T_x\widetilde{V}_j\to\R$ is surjective (so property (ii) holds for $\widetilde{h}_j'$).

 After shrinking $\widetilde{W}'_j$ there exists by Lemma \ref{neighcollar} a strictly positive ${\mathcal S}^\nu$ function $\veps_j:\widetilde{X}_j\to\R$ such that the ${\mathcal S}^\nu$ map 
\[
\widetilde{\phi}_j:=(\widetilde{\rho}_j,\widetilde{h}_j):\widetilde{W}'_j\to \widetilde{\Omega}_j:=\{(x,t) \in \widetilde{X}_j\times\R:\ |t|<\veps_j(x)\}
\] 
is an ${\mathcal S}^\nu$ diffeomorphism and $\veps_j\circ\sigma_j=\veps_j$, so $\sigma_j(\widetilde{W}'_j)=\widetilde{W}'_j$. In addition, as $\widetilde{\phi}_j|_{\widetilde{X}_j}=(\id_{\widetilde{X}_j},0)$, we have $\widetilde{X}_j=\widetilde{h}_j^{-1}(0)$, $d_x\widetilde{h}_j:T_x\widetilde{V}_j\to\R$ is surjective for each $x\in\widetilde{X}_j$ and $\ker(d_x\widetilde{h}_j)=T_x\widetilde{X}_j$. Thus, if $x\in\widetilde{X}_j\cap\widetilde{X}_k$ for some $k\neq j$, then $d_x(\widetilde{h}_j|_{\widetilde{X}_k}):T_x\widetilde{X}_k\to\R$ is also surjective because $\ker(d_x\widetilde{h}_j)=T_x\widetilde{X}_j$ and $\widetilde{X}_j$ and $\widetilde{X}_k$ are transverse at $x$ in $\widetilde{V}_j$. Consequently, by Lemma \ref{neighcollar} we may assume $\widetilde{\phi}_j(\widetilde{W}'_j\cap \widetilde{X}_k)=\widetilde{\Omega}_j \cap ((\widetilde{X}_j\cap \widetilde{X}_k)\times\R)$ for each $k=1,\ldots,s$. Let $\eta_j:\widetilde{X}_j\to\R$ be a strictly positive ${\mathcal S}^\nu$ function such that $\eta_j<\veps_j$ and $\eta_j\circ\sigma_j=\eta_j$ on $\widetilde{X}_j$. Let $f:\R\to[0,1]$ be an even ${\mathcal S}^\nu$ function such that $f(t)=0$ for $|t|\leq\frac{1}{3}$ and $f(t)=1$ for $|t|\geq\frac{1}{2}$. Consider the ${\mathcal S}^\nu$ map
\[
\widetilde{\varphi}_j:\widetilde{\Omega}_j\to\widetilde{\Omega}_j,\ (x,t)\mapsto(x,f(t/\eta_j(x))t).
\]
Define $\widetilde{\psi}_j:=(\widetilde{\phi}_j^{-1}\circ\widetilde{\varphi}_j\circ\widetilde{\phi}_j):\widetilde{W}'_j\to\widetilde{W}'_j$ and observe that $\widetilde{\psi}_j(\widetilde{W}'_j\cap \widetilde{X}_k)\subset \widetilde{W}'_j\cap\widetilde{X}_k$ for each $k=1,\ldots,s$. We extend $\widetilde{\psi}_j$ by the identity to the whole $\widetilde{V}_j$ and obtain a ${\mathcal S}^\nu$ map $\widetilde{\Psi}_j:\widetilde{V}_j\to\widetilde{V}_j$ such that:
\begin{itemize}
\item $\widetilde{\Psi}_j(\widetilde{W}^*_j)=\widetilde{X}_j$ for $\widetilde{W}_j^*:=\widetilde{\phi}_j^{-1}(\{(x,t)\in\widetilde{X}_j\times\R:\ |t|\leq\eta_j(x)/3\})$.
\item $\widetilde{\Psi}_j(y)=y$ if $y\in\widetilde{V}_j\setminus\widetilde{\phi}_j^{-1}(\{(x,t)\in\widetilde{X}_j\times\R:\ |t|<\eta_j(x)/2\})$.
\item $\widetilde{\Psi}_j(\widetilde{X}_k)\subset \widetilde{X}_k$ for each $k=1,\ldots,s$.
\item $\widetilde{\Psi}_j$ is arbitrarily close to the identity on $\widetilde{X}$ if $\eta_j$ is small enough.
\item $\sigma_j\circ\widetilde{\Psi}_j=\widetilde{\Psi}_j\circ\sigma_j$.
\end{itemize}

Only the latter equality $\sigma_j\circ\widetilde{\Psi}_j=\widetilde{\Psi}_j\circ\sigma_j$ requires some comment. It is enough to prove that $\widetilde{\psi}_j=\sigma_j\circ\widetilde{\psi}_j\circ\sigma_j|_{\widetilde{W}_j'}$. Consider the involution 
\[
\tau:\widetilde{\Omega}_j\to\widetilde{\Omega}_j,\ (x,t)\mapsto(\sigma_j(x),-t)
\]
and observe that $\widetilde{\phi}_j\circ\sigma_j|_{\widetilde{W}_j'}=\tau\circ\widetilde{\phi}_j$ and $\tau\circ\widetilde{\varphi}_j\circ\tau=\widetilde{\varphi}_j$ (because $\eta_j\circ\sigma_j=\eta_j$ and $f$ is an even function). Consequently,
\begin{multline*}
\sigma_j\circ\widetilde{\psi}_j\circ\sigma_j|_{\widetilde{W}_j'}=\sigma_j\circ(\widetilde{\phi}_j^{-1}\circ\widetilde{\varphi}_j\circ\widetilde{\phi}_j)\circ\sigma_j|_{\widetilde{W}_j'}\\
=(\widetilde{\phi}_j\circ\sigma_j|_{\widetilde{W}_j'})^{-1}\circ\widetilde{\varphi}_j\circ(\widetilde{\phi}_j\circ\sigma_j|_{\widetilde{W}_j'})=\widetilde{\phi}_j^{-1}\circ\tau\circ\widetilde{\varphi}_j\circ\tau\circ\widetilde{\phi}_j=\widetilde{\phi}_j^{-1}\circ\widetilde{\varphi}_j\circ\widetilde{\phi}_j=\widetilde{\psi}_j.
\end{multline*}
Thus, $\sigma_j\circ\widetilde{\Psi}_j=\widetilde{\Psi}_j\circ\sigma_j$.

We conclude that there exists an ${\mathcal S}^\nu$ map $\Psi_j:M\to M$ such that:
\begin{itemize}
\item $\Psi_j(W_j^*)=X_j$ for $W_j^*:=\widetilde{\pi}_j(\widetilde{W}^*_j)$.
\item $\Psi_j(y)=y$ for each $y\in M\setminus\widetilde{\pi}_j(\widetilde{W}'_j)$.
\item $\Psi_j(X_k)\subset X_k$ for $k=1,\ldots,s$.
\item $\Psi_j$ is arbitrarily close to the identity on $X$ if $\eta_j$ is small enough.
\end{itemize}

\noindent{\sc Final construction.} The composition $\rho:=\Psi_1\circ\cdots\circ\Psi_s$ is an ${\mathcal S}^\nu$ map close to the identity on $X$ and maps the closed semialgebraic neighborhood
\[
W:=\bigcup_{j=1}^s(\Psi_{j+1}\circ\cdots\circ\Psi_s)^{-1}(W_j^*)\subset M
\]
of $X$ onto $X$, where $\Psi_{j+1}\circ\cdots\circ\Psi_s$ denotes $\id_M$ if $j=s$. Thus, $\rho:W\to X$ is a ${\mathcal S}^\nu$ weak retraction that is arbitrarily ${\mathcal S}^0$ close to the identity on $X$, as required.
\end{proof}

\section{Proof of Theorem \ref{main1}}\label{s4}

We are ready to present the proof of Theorem \ref{main1}, which is inspired by some techniques developed in \cite{br}.

\begin{proof}[Proof of Theorem \em\ref{main1}]
First, by \cite[2.7.5]{bcr} we may assume $T$ is closed in $\R^n$. Thus, by \cite{cs}, $T$ is a Nash subset of $\R^n$ (see also \cite{tt}). Let $F\in{\mathcal N}(\R^n)$ be a Nash equation of $T$. By Artin-Mazur description of Nash functions \cite[8.4.4]{bcr} there exists a non-singular irreducible algebraic subset $V$ of some $\R^{n+p}$ of dimension $n$, a connected component $M'$ of $V$, a Nash diffeomorphism $\sigma:\R^n\to M'$ (whose inverse is the restriction to $M'$ of the projection $\Pi:\R^{n+p}\to\R^n$ onto the first $n$ coordinates) and a polynomial function $G:V\to\R$ such that $G(\sigma(x))=F(x)$ for each $x \in \R^n$. In particular, 
\[
\{G=0\}\cap M'=\{F\circ\Pi|_{M'}=0\}=\sigma(T).
\]
Thus, the Zariski closure $\ol{\sigma(T)}^{\zar}$ of $\sigma(T)$ satisfies $\ol{\sigma(T)}^{\zar}\cap M'=\sigma(T)$. Consequently, we may assume from the beginning that $T$ is a finite union of connected components of an algebraic set $Y\subset\R^n$.

Denote $\pi:\R^{n+2}\to\R^{n+1}$ the projection onto the first $n+1$ coordinates. By \cite[Lem.2.2]{br} there exists an irreducible algebraic set $Z\subset\R^{n+2}$ such that $\Sing(Z)=Y\times\{(0,0)\}\subset\{x_{n+1}=0,x_{n+2}=0\}$ and the restriction $\psi:=\pi|_{Z}:Z\to\R^{n+1}$ is a semialgebraic homeomorphism. 

By Theorem \ref{hi1} there exists a non-singular real algebraic set $Z'\subset\R^q$ and a proper regular map $\phi:Z'\to Z$ such that the restriction
\[
\phi|_{Z'\setminus\phi^{-1}(\Sing(Z))}:Z'\setminus\phi^{-1}(\Sing(Z))\to Z\setminus\Sing(Z)
\]
is a Nash diffeomorphism whose inverse map is also regular and $Y':=\phi^{-1}(\Sing(Z))$ is an (algebraic) normal-crossings divisor of $Z'$. As $T':=\phi^{-1}(T\times\{(0,0)\})$ is an open and closed subset of $Y'$, it is a union of connected components of $Y'$, so it is a Nash normal-crossings divisor of $Z'$. 

Let $f\in{\mathcal S}^0(S,T)$, fix a real number $\veps>0$ and let $K_0:=f(S)$, which is a compact semialgebraic subset of $\R^n$. Let $V_1\subset V_2\subset\R^{n+1}$ be open semialgebraic neighborhoods of $K_0\times\{0\}$ whose closures $K_i:=\cl(V_i)$ are compact and $K_1\subset V_2$. As $\psi:Z\to\R^{n+1}$ is a semialgebraic homeomorphism and $\phi:Z'\to Z$ is a proper regular map, we deduce $K'_i:=(\psi\circ\phi)^{-1}(K_i)$ is a compact semialgebraic subset of $Z'$ and $K_1'\subset V_2':=(\psi\circ\phi)^{-1}(V_2)$. The restriction $(\psi\circ\phi)|_{K'_2}:K'_2\to\R^{n+1}$ is a uniformly continuous map, so there exists $\eta>0$ such that if $z,z'\in K'_2$ and $\|z-z'\|_q<\eta$, then 
\begin{equation*}
\|(\psi\circ\phi)(z)-(\psi\circ\phi)(z')\|_{n+1}<\frac{\veps}{3}.
\end{equation*}
By Proposition \ref{wr} there exists a small open semialgebraic neighborhood $W\subset Z'$ of $T'$ and an ${\mathcal S}^\nu$ weak retraction $\rho:W\to T'$ that is arbitrarily ${\mathcal S}^0$ close to the identity on $T'$. Shrinking the Nash manifold $W$, we may assume in addition $W\cap Y'=T'$, $\|\rho(y)-y\|_q<\eta$ for each $y\in W$ and $\rho(\cl(W\cap K'_1))\subset V_2'\subset K_2'$. Thus, if $y\in W\cap K'_1$,
\begin{equation}\label{uco2}
\|(\psi\circ\phi\circ\rho)(y)-(\psi\circ\phi)(y)\|_{n+1}<\frac{\veps}{3}.
\end{equation}

Denote $V'_1:=(\psi\circ\phi)^{-1}(V_1)$. As $\psi\circ\phi$ is proper and $T'\subset W$, the semialgebraic set
\[
C:=(\psi\circ\phi)(Z'\setminus(W\cap V_1'))=(\psi\circ\phi)(Z'\setminus W)\cup(\psi\circ\phi)(Z'\setminus V'_1)
\] 
is a closed semialgebraic subset of $\R^{n+1}$ that does not meet $(T\times\{0\})\cap V_1$. 

Suppose by contradiction that $y\in C\cap(T\times\{0\})\cap V_1$. There exists $z\in Z'\setminus(W\cap V_1')$ such that $(\psi\circ\phi)(z)=y$, so $z\in(\psi\circ\phi)^{-1}(T\times\{0\})\cap(\psi\circ\phi)^{-1}(V_1)=T'\cap V_1'\subset W\cap V_1'$, which is a contradiction. 

Consider the distance function $\delta:\R^{n+1}\to[0,+\infty),\ y\mapsto\dist(y,C)$ and observe that $\delta|_{(T\times\{0\})\cap V_1}$ is strictly positive. As $S$ is compact, the homomorphism $\delta_*:{\mathcal S}^0(S,\R^{n+1})\to{\mathcal S}^0(S,\R),\ g\mapsto\delta\circ g$ is by Proposition \ref{comr} continuous. The continuous semialgebraic function $\delta\circ(f,0)$ is strictly positive on $S$ and attains a minimum $\delta_0>0$ over $S$. Let $\veps'\in(0,\veps)$ be such that if $g\in{\mathcal S}^0(S,\R^{n+1})$ and $\|g-(f,0)\|_{n+1}<\veps'$ on $S$ then 
\[
|\delta\circ g-\delta\circ(f,0)|<\frac{\delta_0}{2} \quad \text{on $S$},
\]
so in particular $\delta\circ g$ is strictly positive and $\im(g)\subset\R^{n+1}\setminus C$. Consider the continuous semialgebraic function $f':=(f,\frac{\veps'}{3}):S\to\R^{n+1}$. We have 
\begin{equation}\label{uco0}
\|(f,0)-f'\|=\frac{\veps'}{3}<\frac{\veps}{3}\quad \text{on $S$}
\end{equation}
and $\im(f')\cap\{x_{n+1}=0\}=\varnothing$, so $\im(f')\cap(\{x_{n+1}=0\}\cup C)=\varnothing$. The following commutative diagram summarizes the situation we have achieved until the moment.
\[
\xymatrix{
W\ar@<0.6ex>[d]^\rho\ar@{^{(}->}[drr]\\
T'\ar@{^{(}->}[r]\ar@{^{(}->}@<0.6ex>[u]\ar[d]^{\phi}&Y'\ar@{^{(}->}[r]\ar[d]^{\phi}&Z'\ar[d]^{\phi}\\
T\times\{(0,0)\}\ar[d]^\psi\ar@{^{(}->}[r]&Y\times\{(0,0)\}=\Sing(Z)\ar[d]^\psi\ar@{^{(}->}[r]&Z\ar@{^{(}->}[r]\ar@<-0.6ex>@{->}[dr]_\psi&\R^{n+2}\ar[d]^\pi\\
T\times\{0\}\ar@{^{(}->}[r]&Y\times\{0\}\ar@{^{(}->}[rr]&&\R^{n+1}\ar@<-0.6ex>@{->}[ul]_{\psi^{-1}}\\
S\ar[u]_{(f,0)}\ar[rrr]_{f':=(f,\frac{\veps}{3})}&&&\R^{n+1}\setminus(\{x_{n+1}=0\}\cup C)\ar@{^{(}->}[u]
}
\]

The ${\mathcal S}^0$ map $\psi^{-1}\circ f'$ satisfies $\im(\psi^{-1}\circ f')\cap(\Sing(Z)\cup\psi^{-1}(C))=\varnothing$ (recall that $\psi(\Sing(Z))=Y\times\{0\}\subset\{x_{n+1}=0\}$), whereas the ${\mathcal S}^0$ map $f'':=(\phi|_{Z'\setminus Y'})^{-1}\circ\psi^{-1}\circ f':S\to Z'$ is well-defined and $\im(f'')\cap(\psi\circ\phi)^{-1}(C)=\varnothing$. Thus, $\im(f'')\cap(Z'\setminus (W\cap V'_1))=\varnothing$, so $\im(f'')\subset W\cap V'_1$. Write $f'':S\to W\cap V'_1$ and note that $f'=\psi\circ\phi\circ f''$. By \eqref{uco2} 
\begin{equation}\label{uco3}
\|\psi\circ\phi\circ\rho\circ f''-f'\|_{n+1}=\|\psi\circ\phi\circ\rho\circ f''-\psi\circ\phi\circ f''\|_{n+1}<\frac{\veps}{3} \quad \text{on $S$}.
\end{equation}

By \cite[Thm.1]{dk} there exists an open semialgebraic neighborhood $U\subset\R^m$ of $S$ such that $f''$ extends to a continuous semialgebraic map $F'':U\to W\cap V'_1$ between the Nash manifolds $U$ and $W\cap V'_1$. Let $H_0:U\to W\cap V'_1$ be a Nash map close to $F''$ (use \cite[Thm.II.4.1]{sh}). The restriction $h_0:=H_0|_S:S\to W\cap V'_1$ is a Nash map close to $f''$. By Proposition \ref{comr} the homomorphism 
\[
(\psi\circ\phi\circ\rho)_*:{\mathcal S}^0(S,W \cap V'_1)\to{\mathcal S}^0(S,\R^{n+1}),\ g\mapsto(\psi\circ\phi\circ\rho)\circ g
\] 
is continuous with respect the ${\mathcal S}^0$ topology, so $(h,0):=\psi\circ\phi\circ\rho\circ h_0:S\to T\times\{0\}$ is close to $\psi\circ\phi\circ\rho\circ f'':S\to\R^{n+1}$. Thus, we may assume 
\begin{equation}\label{uco4}
\|\psi\circ\phi\circ\rho\circ h_0-\psi\circ\phi\circ\rho\circ f''\|_{n+1}<\frac{\veps}{3}.
\end{equation}
By \eqref{uco0}, \eqref{uco3} and \eqref{uco4} we deduce
\begin{multline*}
\|h-f\|_n=\|(h,0)-(f,0)\|_{n+1}\leq\|\psi\circ\phi\circ\rho\circ h_0-\psi\circ\phi\circ\rho\circ f''\|_{n+1}\\
+\|\psi\circ\phi\circ\rho\circ f''-f'\|_{n+1}+\|f'-(f,0)\|_{n+1}<\frac{\veps}{3}+\frac{\veps}{3}+\frac{\veps}{3}=\veps.
\end{multline*}

In addition, $(h,0)=\psi\circ\phi\circ\rho\circ h_0$ is an ${\mathcal S}^\nu$ map, because it is a composition of ${\mathcal S}^\nu$ maps. Thus, we have found an ${\mathcal S}^\nu$ map $h:S\to T$ that is close to $f$, as required.
\end{proof}


\appendix
\section{Proof of Proposition \ref{comr}}\label{A}

We afford here the proof of Proposition \ref{comr} in the general case.

\begin{proof}[Proof of Proposition \em \ref{comr}]
As $S$ is locally compact, it is closed in some open semialgebraic set $U\subset\R^m$. Analogously, $T$ is closed in some open semialgebraic set $V\subset\R^n$ and there exists a ${\mathcal S}^\nu$ map $F:V\to\R^p$ that extends $f$. For each open semialgebraic neighborhood $U'\subset U$ of $S$ we have the following commutative diagram:
\begin{equation}\label{diag2}
\xymatrix{
&{\mathcal S}^\nu(S,T)\ar[r]^{f_*}\ar@{^{(}->}[d]^{\tt j}&{\mathcal S}^\nu(S,T')\ar@{^{(}->}[d]^{{\tt j}'}\\
{\mathcal S}^\nu(S,\R^n)&{\mathcal S}^\nu(S,V)\ar[r]^{F_*}\ar@{_{(}->}[l]^{\tt i}&{\mathcal S}^\nu(S,\R^p)\\
{\mathcal S}^\nu(U',\R^n)\ar[u]_{\eta_{U'}}&{\mathcal S}^\nu(U',V)\ar[u]_{\theta_{U'}}\ar[r]^{F_*}\ar@{_{(}->}[l]^{{\tt i}'}&{\mathcal S}^\nu(U',\R^p)\ar[u]_{\rho_{U'}}
}
\end{equation}
where $\theta_{U'}$, $\eta_{U'}$ and $\rho_{U'}$ denote the (continuous) restriction homomorphisms. Note that $\eta_{U'}$ and $\rho_{U'}$ are actually open (surjective) quotient maps. We claim: 
\[
{\mathcal S}^\nu(S,V)=\bigcup_{\substack{\text{$U'$ open in $U$,}\\{S\subset U'}}}\theta_{U'}({\mathcal S}^\nu(U',V)).
\]
The inclusion right to left is clear, so let us prove the converse one. Let $g\in{\mathcal S}^\nu(S,V)$ and let $G:U\to\R^n$ be a ${\mathcal S}^\nu$ extension of $g$ to $U$. Let $U':=G^{-1}(V)$ and $G':=G|_{U'}:U'\to V$, which is a ${\mathcal S}^\nu$ extension of $g$ to $U'$ whose image is contained in $V$, that is, $G'\in{\mathcal S}^\nu(U',V)$.

The lower $F_*$ in diagram \eqref{diag2} is continuous (it is the Nash manifolds case \cite[II.1.5, p.\hspace{1.5pt}83]{sh}), so the composition $\rho_{U'}\circ F_*$ is continuous too for each open semialgebraic neighborhood $U'\subset U$ of $S$. Let us prove: \em The map $F^*$ in the middle row of diagram \eqref{diag2} is continuous\em. 

Let ${\mathcal W}$ be an open subset of ${\mathcal S}^\nu(S,\R^p)$. We first show
\[
F_*^{-1}({\mathcal W})=\bigcup_{\substack{\text{$U'$ open in $U$,}\\{S\subset U'}}}\theta_{U'}((\rho_{U'}\circ F_*)^{-1}({\mathcal W})).
\]
Indeed, let $g\in F_*^{-1}({\mathcal W})$ and let $U'_0\subset U$ be an open neighborhood of $S$ for which there exists $G\in{\mathcal S}^\nu(U'_0,V)$ with $\theta_{U'_0}(G)=g$. Observe that $(\rho_{U'_0}\circ F_*)(G)=(F_*\circ \theta_{U_0'})(G)=F_*(g)\in{\mathcal W}$. Thus, $g=\theta_{U'_0}(G)\in\theta_{U'_0}((\rho_{U'_0}\circ F_*)^{-1}({\mathcal W}))$. The converse inclusion is clear. 

Hence it suffices to show that \em each set $\theta_{U'}((\rho_{U'}\circ F_*)^{-1}({\mathcal W}))$ is open in ${\mathcal S}^\nu(S,V)$\em. As $\rho_{U'}\circ F^*$ is continuous, ${\tt i}$ is continuous, ${\tt i}'$ is a homeomorphism onto its image, the equality ${\tt i}\circ\theta_{U'}=\eta_{U'}\circ{\tt i}'$ holds and $\eta_{U'}$ is an open quotient map for each open neighborhood $U'\subset U$ of $S$, we are reduced to prove: \em ${\mathcal S}^\nu(U',V)$ is an open subset of ${\mathcal S}^\nu(U',\R^n)$ for each open neighborhood $U'\subset U$ of $S$\em, which follows from \cite[2.D]{bfr}. Consequently $F^*$ is continuous.

Now we turn to the upper square of diagram \eqref{diag2}. As $F_*$ in the middle row is continuous, the composition $F_*\circ{\tt j}$ is continuous too. But this map coincides with ${\tt j}'\circ f_*$, which is consequently continuous. As ${\tt j}'$ is a homeomorphism onto its image, $f_*$ is continuous, as required.
\end{proof}


\bibliographystyle{amsalpha}

\end{document}